\documentclass[12pt]{amsart}
 \usepackage{amsfonts,amssymb}

\usepackage{mathrsfs}
\let\mathcal\mathscr

\usepackage{array}

\usepackage[all,ps,cmtip]{xy} 

\def\Z{{\bf Z}}

\def\C{{\bf C}}

\def\R{{\bf R}}
\def\Q{{\bf Q}}
\def\P{{\bf P}}
\def\T{{\bf T}}

\def\phi{{\varphi}}

\def\cI{\mathcal{I}}
\def\cD{\mathcal{D}}

\def\cO{\mathcal{O}}

\def\cE{\mathcal{E}}

\def\cC{\mathcal{C}}

\def\cQ{\mathcal{Q}}
\def\cU{\mathcal{U}}

\def\cX{\mathcal{X}}

\def\lra{\longrightarrow}
\def\llra{\hbox to 10mm{\rightarrowfill}}
\def\lllra{\hbox to 15mm{\rightarrowfill}}

\def\llla{\hbox to 10mm{\leftarrowfill}}
\def\lllla{\hbox to 15mm{\leftarrowfill}}

\def\dra{\dashrightarrow}
\def\thra{\twoheadrightarrow}
\def\hra{\hookrightarrow}
\def\isom{\simeq}
\def\eps{\varepsilon}
\def\ie{\hbox{i.e.}}
\def\eg{\hbox{e.g.}}

 \def\vide{\varnothing}

\DeclareMathOperator{\isomto}{\stackrel{{}_{\scriptstyle\sim}}{\to}}

\DeclareMathOperator{\Aut}{Aut}

\DeclareMathOperator{\Def}{Def}

\def\div{\mathop{\rm div}\nolimits}

\DeclareMathOperator{\Hom}{Hom}
\DeclareMathOperator{\Id}{Id}

\DeclareMathOperator{\Ker}{Ker}

\DeclareMathOperator{\lin}{\underset{\mathrm lin}{\equiv}}

\DeclareMathOperator{\PGL}{PGL}
\DeclareMathOperator{\Pic}{Pic}

\DeclareMathOperator{\SO}{SO}

\DeclareMathOperator{\Sing}{Sing}

\DeclareMathOperator{\Sym}{Sym}

\def\llra{\hbox to 10mm{\rightarrowfill}}
\def\lllra{\hbox to 15mm{\rightarrowfill}}

\newtheorem{lemm}{Lemma}[section]
\newtheorem{theo}[lemm]{Theorem}
\newtheorem{coro}[lemm]{Corollary}
\newtheorem{prop}[lemm]{Proposition}

\theoremstyle{definition}

\newtheorem{rema}[lemm]{Remark}

\newtheorem{ques}[lemm]{Question}

\theoremstyle{remark}
\newtheorem*{remark*}{Remark}
\newtheorem*{note*}{Note}

\def\ra{\to}

\begin{document}
\title[Special Fano fourfolds]{Special prime Fano fourfolds of degree 10 and index 2}

\author[O. Debarre]{Olivier Debarre}
\thanks{O. Debarre and L. Manivel are part of the  project VSHMOD-2009   ANR-09-BLAN-0104-01.}
\address{D\'epartement de  Math\'ematiques et Applications,
\'Ecole normale sup\'erieure -- CNRS,
45 rue d'Ulm, 75230 Paris cedex 05, France}
\email{{\tt olivier.debarre@ens.fr}}
\author[A. Iliev]{Atanas Iliev}
\address{Department of Mathematics, Seoul National University, Gwanak Campus, Bldg. 27 Seoul 151-747, Korea}
\email{{\tt ailiev2001@yahoo.com}}
\author[L. Manivel]{Laurent Manivel}
\address{Institut Fourier,  
Universit\'e de Grenoble I -- CNRS,
BP 74, 38402 Saint-Martin d'H\`eres, France}
\email{{\tt laurent.manivel@ujf-grenoble.fr}}

\def\moins{\mathop{\hbox{\vrule height 3pt depth -2pt
width 5pt}\,}}
\def\pmoins{\mathop{\hbox{\vrule height 2.1pt depth -1.3pt
width 4pt\,}}}

\date{\today}

\keywords{Fano fourfolds, Torelli problem, lattices, period domains, unirationality, rationality, special fourfolds, Noether-Lefschetz locus, K3 surfaces, cubic fourfolds.}
\subjclass{14C25, 14C30, 14C34, 14D20, 14D23, 14E05, 14E08, 14E30, 14J10, 14J28, 14J35, 14J45, 14J70, 14M15}

 \begin{abstract}
 We analyze (complex) prime Fano fourfolds of degree 10 and index 2. Mukai gave a complete
geometrical description; in particular, most of them are contained in a Grassmannian $G(2,5)$. They are all unirational and, as in the case of cubic fourfolds, some are rational, as already remarked by Roth in 1949.

We show that their middle cohomology is of K3 type and that their period map is dominant, with smooth 4-dimensional fibers, onto a   20-dimensional bounded 
symmetric period domain of type IV. Following Hassett, we say that such a fourfold is {\em special} if it contains a surface whose cohomology class does not come from the Grassmannian $G(2,5)$. Special fourfolds correspond to a countable union of   hypersurfaces (the Noether-Lefschetz locus) in the period domain, labelled by a positive integer $d$. We describe special fourfolds for some low values of $d$. We also characterize those integers $d$ for which special fourfolds  do exist.
 \end{abstract}
 
\maketitle


 \section{Introduction}
 
 One of the most vexing classical questions in complex algebraic geometry is whether there exist  irrational smooth cubic hypersurfaces  in $\P^5$. They are all unirational, and rational examples are easy to construct (such as Pfaffian  cubic fourfolds)  but no smooth cubic fourfold has yet been proven to be 
irrational. The general feeling seems to be that   the question should have an   affirmative answer 
but, despite numerous attempts, it is still open. 

In a couple of very interesting   articles on cubic fourfolds (\cite{has}, \cite{hasrat}), Hassett adopted a Hodge-theoretic approach and,   using the  period map (proven to be injective by Voisin in \cite{voi}) and the geometry of the  period domain, a 20-dimensional bounded 
symmetric domain of type IV,  he related geometrical properties of a cubic fourfold to arithmetical properties of its period point. 

We do not solve the rationality question  in this article, but we investigate instead similar questions for another family  of Fano fourfolds   (see \S\ref{se2} for their definition). Again, they are all unirational (see \S\ref{suni}), and rational examples were found by Roth (see \cite{rot}, and also \cite{pro}  and \S\ref{exam}), but no irrational examples are known.

We prove in \S\ref{scoho} that the moduli stack $\cX_{10}$ associated with these fourfolds  is smooth of dimension 24 (Proposition \ref{hitx}) and that the period map is smooth and dominant onto, again,  a 20-dimensional bounded 
symmetric domain of type IV (Theorem \ref{sub}). We identify the underlying lattice in \S\ref{spp}. Then, following  \cite{has}, we define in \S\ref{sl} hypersurfaces  in the period domain which parametrize ``special'' fourfolds $X$,   whose period point satisfy a non-trivial arithmetical property depending on a positive integer $d$, the {\em discriminant.} As in \cite{has}, we characterize in Proposition \ref{p54} those integers $d$ for which the non-special cohomology of a special $X$ is essentially the primitive cohomology of a K3 surface; we say that this K3 surface is {\em associated} with $X$. Similarly, we characterize in Proposition \ref{p55} those $d$ for which the non-special cohomology of a special $X$ is the non-special cohomology of a cubic fourfold in the sense of \cite{has}.

In \S\ref{exam}, we  give  geometrical constructions    of    special fourfolds for   $d\in\{8,10,12\}$; in particular, we discuss some rational examples (already present in \cite{rot} and \cite{pro}). When $d=10$, the associated K3 surface (in the sense of Proposition \ref{p54}) does appear in the construction; when $d=12$, so does the associated cubic fourfold (in the sense of Proposition \ref{p55})  and they are even birationally isomorphic.

In \S\ref{cons}, we  characterize the positive integers $d$ for which there exist  (smooth) special fourfolds of discriminant $d$. As in \cite{has}, our construction relies on the surjectivity of the period map for K3 surfaces. Finally, we summarize in \S\ref{last} some of our results on the period map and ask a couple of questions about its geometry.

So in some sense, the picture is very   similar to what we have for cubic fourfolds, with one big difference: the Torelli theorem does not hold. In a forthcoming article, building on the link between our fourfolds and EPW sextics discovered in \cite{im}, we will analyze more closely the (4-dimensional) fibers of the period map.

 \medskip
\noindent{\bf Acknowledgements.}
 We thank  V. Gritsenko and 
 G. Nebe for their help with lattice theory,  
 and K. Ranestad   and  C.-L. Wang for drawing our attention to the references \cite{jk} and \cite{fuwang} respectively.
 
  \section{Prime Fano fourfolds of degree 10 and index 2}\label{se2}
 
Let $X$ be a   (smooth) prime  Fano fourfold  of degree 10 (\ie, of ``genus'' 6) and index 2; this means that
$\Pic(X)$ is generated by the class of an ample divisor $H$ such that $H^4=10$ and $-K_X\lin 2H$. 
Then $H$ is very ample and embeds $X$ in $\P^8$  as follows (\cite{muk3}; \cite{ip}, Theorem 5.2.3).

   Let    $V_5$ be a $5$-dimensional   vector space (our running notation is $V_k$ for any $k$-dimensional vector space). Let $G(2,V_5) \subset \P(\wedge^2 V_5)$ be the   Grassmannian in its Pl\"ucker embedding and let   $CG\subset \P(\C\oplus \wedge^2V_5)\isom\P^{10}$ be the cone, with vertex $v=\P(\C)$, over  $G(2,V_5)$. Then
 $$X=CG\cap \P^8\cap Q,
 $$
 where   $ Q$ is a quadric. 
  There are two cases:
\begin{itemize}
\item either $v\notin\P^8$, in which case  $X$ is isomorphic to the intersection   of $G(2,V_5)\subset \P(\wedge^2 V_5)$ with a hyperplane (the projection of $\P^8$ to $ \P( \wedge^2V_5)$) and a quadric;
\item or $v\in\P^8$, in which case $\P^8$ is a cone over a $\P^7\subset   \P( \wedge^2V_5)$ and   $X$ is a double cover   of $ G(2,V_5)\cap \P^7$ branched along its intersection with a quadric.
\end{itemize}
The varieties obtained by the second construction will be called ``of Gushel type'' (after Gushel, who studied the 3-dimensional case in \cite{gus}). They are specializations of varieties obtained  by the first construction.

Let $\cX_{10}$ be the irreducible moduli stack 
for (smooth) prime  Fano fourfolds  of degree 10 and index 2,  let $\cX_{10}^G$ be the (irreducible closed) substack of those which are of Gushel type, and let $\cX_{10}^0:=\cX_{10}\moins \cX_{10}^G$. We have
$$\dim(\cX_{10})=24\quad ,\quad\dim(\cX_{10}^G)=22.
$$

\section{Unirationality}\label{suni}

 Let $G := G(2,V_5)$ and let $X := G\cap \P^8\cap  Q$ be a fourfold of type $\cX^0_{10}$. We give a new proof of the classical fact that $X$ is unirational.

 The hyperplane $\P^8$ is defined by  a non-zero skew-symmetric form $\omega$ on $V_5$, and the singular locus of $G^\omega:=G\cap \P^8$ is isomorphic to $G(2,\Ker(\omega))$. Since $X$ is smooth, this singular locus must be finite, hence $\omega$ must be of maximal rank and $G^\omega$ is also smooth.   The variety $G^\omega$ is the unique del Pezzo fivefold of degree 5 (\cite{ip}, Theorem 3.3.1); it  parametrizes isotropic 2-planes for the   form   $\omega $.

If $V^\omega_1\subset V_5$ is the kernel   of  $\omega$, the 3-plane $\P^3_0:=\P(V^\omega_1\wedge V_5)$ of lines passing through $[V^\omega_1]\in \P(V_5)$ is contained in $ G^\omega$, hence $X$ contains  $\Sigma_0:=\P^3_0\cap Q$, a ``$\sigma$-quadric'' surface,\footnote{This means that the lines  in $\P(V_5)$ parametrized by $\Sigma_0$ all pass through a fixed point. Since $X$ is smooth, it contains no 3-planes by the Lefschetz theorem, hence $\Sigma_0$ is indeed a surface.} possibly reducible.
It is the only   $\sigma$-quadric contained in $X$.\footnote{If $\Sigma\subset X$ is a $\sigma$-quadric, its span $\P(V'_1\wedge V_5)$ is contained in the isotropic Grassmannian $G_\omega(2,V_5)$, hence $V'_1$ is the kernel of $\omega$ and $\Sigma=\Sigma_0$.}

 \begin{prop}\label{unir}
 Any   fourfold $X$ of type $\cX^0_{10}$ is unirational. More precisely, there is a rational double cover $\P^4\dra X$.
 \end{prop}
 
 \begin{proof}
 If $p\in \Sigma_0$, the associated $V_{2,p}\subset V_5$ contains ${V_1^\omega}$, hence its $\omega$-orthogonal is a hyperplane $V^\bot_{2,p}\subset V_5$. The $(\P^1\times\P^1)$-bundle  $Y:=\bigcup_{p\in \Sigma_0} \P(V_{2,p})\times \P(V^\bot_{2,p}/V_{2,p})$ over $\Sigma_0$ is then a rational fourfold.

A general point of $Y$ defines a   flag ${V_1^\omega}\subset V_{2,p}\subset V_3\subset V^\bot_{2,p}\subset V_5$ hence a
  line in $G(2,V_5)$ passing through $p$ and contained in $\P^8$. This line meets $X\moins \Sigma_0$ at a unique point, and this defines a rational map $Y\dra X$.
 
 This map has degree 2: if $x$ is general in $X$,  lines in $G(2,V_5)$ through $x$ meet $\P({V_1^\omega}\wedge V_5)$ in points $p$ such that $V_{2,p}={V_1^\omega}+V_1$, with $V_1\subset V_{2,x}$, hence the intersection is $\P({V_1^\omega}\wedge V_{2,x})$. This is a line, therefore it   meets $\Sigma_0$ in  two points.  \end{proof}
 
 As Kuznetsov remarked, the unirationality construction above can be made more precise by showing that 
the blow-up of $X$ along the surface $\Sigma_0$ can be expressed as a conic bundle over the 3-dimensional quadric $G(2,V_5/{V_1^\omega})\cap\P^8$ and that the exceptional divisor is rational and maps 2-to-1 onto the base.

 \section{Cohomology and the local period map}\label{scoho}

As in \S\ref{suni}, set   $G := G(2,V_5)$ and  let ${G^\omega}:=G\cap \P^8$ be a smooth hyperplane section of $G$.

\subsection{The Hodge diamond of $X$}

 The inclusion ${G^\omega}\subset G$ induces isomorphisms 
 \begin{equation}\label{hodw5} 
 H^k(G,\Z)\isomto H^k({G^\omega},\Z)\qquad\text{for all }k\in\{0,\dots,5\}. 
  \end{equation}
  The Hodge diamond for a fourfold  $X := {G^\omega}\cap  Q$ of type $\cX^0_{10}$   was computed in \cite{im}, Lemma 4.1; its upper half is as follows:
 \begin{equation}\label{hodx4}
 \begin{array}{ccccccccc}
 &&&& 1   \\
 &&& 0 && 0  \\
 &&0&& 1 &&0   \\
 &0&& 0 && 0 &&0  \\
0 &&1&& 22 &&1 &&0 
\end{array} 
 \end{equation}
When $X$ is of Gushel type, the Hodge diamond remains the same. In all cases, the rank-2 lattice $H^4(G,\Z)$ embeds into $ H^4(X,\Z)$ and {\em  we define the vanishing cohomology $H^4(X,\Z)_{\rm van}$   as the orthogonal} (for the intersection form) {\em  of the image of $ H^4(G,\Z) $ in $H^4(X,\Z)$.} It is a lattice of rank 22.

 \subsection{The local deformation space}
 
We compute the cohomology groups of the tangent sheaf $T_X$.

 \begin{prop}\label{hitx}
For any    fourfold $X$ of type $\cX_{10}$, we have
$$H^p(X,T_X)=0 \qquad \hbox{\it for}\quad p\ne 1 $$
and $h^1(X,T_X)=24$. In particular, the group of automorphisms of $X$ is finite and the local deformation space $\Def(X)$ is smooth of dimension 24.
\end{prop}

 \begin{proof}
For $p\ge 2$, the conclusion follows from the Kodaira-Akizuki-Nakano   
  theorem, since
  $T_X\isom \Omega^3_X(2) $. 
  
 {\em Assume first that $X$ is {\em not} of Gushel type,} so that  $X  = {G^\omega}\cap  Q$.
  
Let us prove $H^0(X,T_X)=0$. We have inclusions $X\subset {G^\omega}\subset G$. The conormal exact sequence $0\to \cO_X(-2)\to \Omega^1_{G^\omega}\vert_X\to \Omega^1_X\to 0$ induces an exact sequence
 $$0\to \Omega^2_X\to \Omega^3_{G^\omega}(2)\vert_X\to T_X\to 0.$$
  Since  $H^1(X,\Omega^2_X)$  vanishes,  it is enough to show $H^0(X,\Omega^3_{G^\omega}(2)\vert_X)=0$. Since $H^1({G^\omega},\Omega^3_{G^\omega}  )= 0$, it is enough to show  that $H^0({G^\omega},\Omega^3_{G^\omega}(2) )$, or equivalently its Serre-dual $H^5({G^\omega},\Omega^2_{G^\omega}(-2) )$, vanishes.
  
  The conormal exact sequence of ${G^\omega}$ in $G$   induces an exact sequence
 $$0\to \Omega^1_{G^\omega}(-3)\to \Omega^2_G(-2)\vert_{G^\omega}\to \Omega^2_{G^\omega}(-2)\to 0.$$
 The desired vanishing follows since $H^5(G, \Omega^2_G(- 2))=H^6(G, \Omega^2_G(- 3))=0$ by Bott's theorem. 

 {\em Assume now that $X$   {\em is} of Gushel type,} so we have a double cover $\pi_X:X\to W:=G\cap \P^7$ branched along the intersection of $W$ with a quadric.
 We have an exact sequence
  \begin{equation}\label{xwd}
  0\to T_X\to \pi_X^*T_{W}\to \cO_R(R)\to 0,\end{equation}
  where $R\subset X$ is the ramification of $\pi_X$. 
 Since $\Omega^1_G(-2)\vert_{W}$ is acyclic (\cite{dim}, proof of Proposition 3.3), we have (\cite{dim}, (4.7))
 $$H^0({W}, T_{W})\isom H^4({W}, \Omega^1_{W}(-2))^\vee=0.$$
 Since $H^0(X,\pi_X^*T_{W})\isom H^0({W}, T_{W})\oplus H^0({W}, T_{W}(-1))$, we obtain the desired vanishing in the Gushel case.
  
 Since $X$ is (anti)canonically polarized, this vanishing implies that its group of automorphisms is a discrete subgroup of $\PGL(9,\C)$, hence is finite.
 Finally, we leave the computation of
$h^1(X,T_X)=-\chi(X,T_X)=-\chi(X,\Omega^1_X(-2))$  to the reader.  
\end{proof}

 \begin{rema}When $X$    is  of Gushel type, we have with the notation above  $ H^1(R,\cO_R(2))=0 $    by Kodaira vanishing, hence  an exact sequence
 $$ 0  \to   H^0(R,\cO_R(2)) \to   H^1(X,T_X) \to   H^1({W}, T_{W})\oplus H^1({W}, T_{W}(-1)) \to   0 .$$Moreover,$$ H^1({W}, T_{W})\isom H^3({W}, \Omega^1_{W}(-3))^\vee =0 .$$ Similarly,  $H^1({W}, T_{W}(-1))\isom L^\vee$,  where   $L\subset \wedge^2V_5^\vee$ is the 2-dimensional vector space that defines the $\P^8=L^\bot\subset \wedge^2V_5 $ that defines $W $. The kernel of the map $ H^1(X,T_X) \to     H^1({W}, T_{W}(-1))$ describes the tangent space to the Gushel locus.\end{rema}

 \subsection{The local period map}
 
  Let $X$ be a  fourfold of type $\cX_{10}$  and let $\Lambda$ be a fixed lattice isomorphic to  $ H^4(X,\Z)_{\rm van}$. 
By Proposition \ref{hitx}, $X$ has a smooth (simply connected) local deformation space $\Def(X)$ of  dimension 24. 
By (\ref{hodx4}), the Hodge structure of $H^4(X)_{\rm van}$ is of K3 type hence we can define a morphism
 $$\Def(X)\to \P(\Lambda\otimes\C)$$
 with values in the smooth 20-dimensional quadric
 $$\cQ:=\{\omega\in \P(\Lambda\otimes\C)\mid   (\omega\cdot \omega) = 0\} 
.
 $$
We show below  (Theorem \ref{sub}) that the restriction $p:\Def(X)\to\cQ$, the {\em local period map,}  is a submersion.

Recall from \S\ref{suni}  that the hyperplane $\P^8$ is defined by a skew-symmetric form on
$V_5$ whose kernel is a one-dimen\-sional subspace ${V_1^\omega}$ of $V_5$. 

\begin{lemm}\label{le41} There is an isomorphism
 $H^1({G^\omega},\Omega_{G^\omega}^3(2))\simeq V_5/{V_1^\omega}$.
\end{lemm}

\begin{proof} From the normal exact sequence 
of the embedding ${G^\omega}\subset G $, we deduce the exact sequences
\begin{eqnarray}
&0\to \Omega_{G^\omega}^1\to\Omega_G^2(1)\vert_{G^\omega}\to\Omega_{G^\omega}^2(1)\to 0&\label{equ1}\\
&0\to \Omega_{G^\omega}^2(1)\to\Omega_G^3(2)\vert_{G^\omega}\to\Omega_{G^\omega}^3(2)\to 0.&\label{equ2}
\end{eqnarray}
By Bott's theorem, $\Omega_G^2(1)$ is acyclic, so that we have
$$H^q({G^\omega}, \Omega_G^2(1)\vert_{G^\omega})\isom H^{q+1}(G,\Omega_G^2)\isom \delta_{q,1}\C^2.$$
On the other hand,  by (\ref{hodw5}), we have $H^q({G^\omega}, \Omega_{G^\omega}^1)\isom \delta_{q,1}\C$.
Therefore, we also get, by (\ref{equ1}), $H^q({G^\omega}, \Omega_{G^\omega}^2(1)) \simeq\delta_{q,1} {V_1^\omega}.$

By Bott's theorem again, $\Omega_G^3(1)$ is acyclic  hence, using (\ref{equ2}), we obtain
$$H^q({G^\omega},\Omega_G^3(2)\vert_{G^\omega})\isom H^q(G,\Omega_G^3(2))\isom\delta_{q,1}V_5.$$
This finishes the proof of the lemma. \end{proof}

\begin{theo}\label{sub}
For any fourfold $X$ of type $\cX_{10}$,   the local 
period map $p:\Def(X)\to \cQ$  is a submersion.
\end{theo}

\begin{proof}The tangent map to $p$ at the point $[X]$ defined by $X$ has same kernel as the morphism 
\begin{eqnarray*}
 T: H^1(X,T_X)&\to&  \Hom ( H^{3,1}(X),H^{3,1}(X)^\bot/H^{3,1}(X))\\
  & \isom&\Hom ( H^1(X,\Omega^3_X),H^2(X,\Omega^2_X))
  \end{eqnarray*}
defined by the natural pairing $H^1(X,T_X)\otimes H^1(X,\Omega^3_X)\to
H^2(X,\Omega^2_X)$ (by (\ref{hodx4}), $ H^1(X,\Omega^3_X)$ is one-dimensional).  

 {\em Assume first that $X$ is {\em not} of Gushel type,} so that, keeping the notation above, $X$ is a smooth quadratic section of ${G^\omega}$.
Recall the 
 isomorphism  $T_X\isom \Omega_X^3(2)$. The normal exact sequence 
of the embedding $X\subset {G^\omega}$ yields the exact sequence
$0\ra \Omega_X^2\ra\Omega_{G^\omega}^3(2)\vert_X\ra T_X\ra 0$. 

Moreover, the induced coboundary map 
$$H^1(X,T_X)\ra H^2(X,\Omega_X^2)$$
coincides with the cup-product by a generator of $H^1(X,\Omega_X^3)\simeq\C$, hence is the morphism $T$.
 Since $H^{2,1}(X)=0$ (see (\ref{hodx4})), its kernel $K$ is isomorphic to $H^1(X,\Omega_{G^\omega}^3(2)\vert_X)$.
 
In order to compute this cohomology group, we consider the exact 
sequence 
  $0\ra \Omega_{G^\omega}^3\ra\Omega_{G^\omega}^3(2)\ra \Omega_{G^\omega}^3(2)\vert_X\to 0$. 
Since, by (\ref{hodw5}), we have $H^1({G^\omega},\Omega_{G^\omega}^3)=H^2({G^\omega},\Omega_{G^\omega}^3)=0$, we get  
$$K\isom H^1(X,\Omega_{G^\omega}^3(2)\vert_X)
\isom H^1({G^\omega},\Omega_{G^\omega}^3(2))\isom V_5/{V_1^\omega}$$ by  Lemma \ref{le41}. 
Since $\Def(X)$ is smooth of dimension 24 and $\cD$ is smooth of dimension 20, this concludes the proof of the theorem in this case.

  {\em Assume now that $X$ is  of Gushel type,} so we have a double cover $\pi_X:X\to W:=G\cap \P^7$ branched along the intersection of $W$ with a quadric.  We   consider $X$ as a subvariety of the blow-up $\P W$ of the vertex of the (projective) cone   over $W$. The  $\P^1$-bundle $\pi:\P W\to W$ is associated with $\cE=\cO_{W}\oplus \cO_{W}(1)$ (in Grothendieck's notation). We have $\omega_{\P {W}/{W}}\isom \pi^*\det (\cE)\otimes \cO_{\P {W}}(-2) $, hence $\pi_*\omega_{\P {W}/{W}}=0$, while  $R^1\pi_*\omega_{\P {W}/{W}}=\cO_{W}$ by Grothendieck's duality. The exact sequence
 \begin{equation}\label{pw}
 0\to \pi^*\Omega_{W}^3\to \Omega_{\P {W}}^3\to \pi^*\Omega_{W}^2\otimes \omega_{\P {W}/{W}}\to 0
 \end{equation}
   gives a long exact sequence$$\cdots\to H^i({W} , \Omega_{W}^3)\to H^i(\P {W} , \Omega_{\P {W}}^3)\to H^{i-1}({W} , \Omega_{W}^2)\to \cdots$$
   In particular, we obtain, using the fact that $h^{p,q}(W)=0$ for all $p\ne q$,
   $$H^1({\P {W}},\Omega_{\P {W}}^3)=H^2({\P {W}},\Omega_{\P {W}}^3)=0.$$
   As in the first case above, the kernel $K$ is therefore isomorphic to $H^1({\P {W}},\Omega_{\P {W}}^3\otimes\cO_{\P {W}}(2))$. Using (\ref{pw}) again, we obtain an exact sequence 
   $$H^1({W} , \Omega_{W}^3\otimes \Sym^2\cE)\to K\\\to H^1({W} , \Omega_{W}^2\otimes \det (\cE)).$$
   By \cite{dim}, Proposition 5.3, we have
   $$\begin{array}{rcccccl}H^1({W} , \Omega_{W}^3\otimes \Sym^2\cE)&=&H^1({W} , \Omega_{W}^3(2))&\isom&H^3({W} , \Omega_{W}^1(-2))^\vee,\\H^1({W} , \Omega_{W}^2\otimes \det (\cE))&=&H^1({W} , \Omega_{W}^2(1))&\isom&H^3({W} , \Omega_{W}^2(-1))^\vee\end{array}$$
   and these spaces are both 2-dimensional. In particular, $K$ has dimension at most 4. By semi-continuity, there is equality and the theorem is proved.\end{proof}

  If $X$ is of Gushel type, we may also consider, inside $\Def(X)$, the locus $\Def^G(X)$ where the deformation of $X$ remains  of Gushel type and the restriction 
  $$p^G:\Def^G(X)\to \cQ$$
  of the local period map.

 The tangent space to $\Def^G(X)$  at $[X]$ corresponds to the subspace of $H^1(X,T_X)$ where the Gushel involution acts trivially. One can describe it as follows: the inclusion $T_X\hra \pi_X^*T_W$ induces a map
 $$\begin{array}{rcccl}H^1(X,T_X)&\to& H^1(X, \pi_X^*T_W)&\isom& H^1(W,T_W)\oplus H^1(W,T_W(-1))\\&&&=& H^1(W,T_W(-1))\\&&&\isom&  H^1(W , \Omega_W^3(2))\end{array}$$
 which can be checked, using (\ref{xwd}), to be surjective. Its kernel   is the tangent space to the Gushel locus, and it follows from the proof above that the intersection of $K=\Ker(T_{p,[X]})$ with that space has dimension 2.

 \begin{prop}For any smooth $X$ of type $\cX^G_{10}$, the kernel of $T_{p^G,[X]}$ is 2-dimensional. In particular, $p^G$ is a submersion at $[X]$.\end{prop}

The fact that the period map is dominant implies a Noether-Lefschetz-type result.

\begin{coro}\label{simple}
If $X$ is a very general fourfold   of type $\cX_{10}$, or is very general of type $\cX^G_{10}$, we have $ H^{2,2}(X) \cap H^4(X,\Q)=H^4(G,\Q)  $ and  the Hodge structure $H^4(X,\Q)_{\rm van}$ is simple.
\end{coro}

\begin{proof}
For $ H^{2,2}(X)\cap H^4(X,\Q)_{\rm van}$ to be non-zero, the corresponding period must be in one of the (countably many) hypersurfaces    $\alpha^\bot \cap\cQ$, for some $\alpha \in \P(\Lambda\otimes\Q  )$. Since the local period map is dominant, this does not happen for $X$ very general (or very general of Gushel type). 

For any $X$, a standard argument (see, \eg, \cite{zar}, Theorem 1.4.1)
 shows that the transcendental lattice  $\bigl( H^4(X,\Z)_{\rm van}\cap H^{2,2}(X)\bigr)^\bot$ inherits a simple rational Hodge structure. For $X$ very general (or very general of Gushel type), the transcendental lattice is 
$H^4(X,\Z)_{\rm van}$.
\end{proof}

   \section{The period domain and the period map}\label{spp}

\subsection{The vanishing cohomology lattice}\label{s41}
Let $(L,\cdot)$ be a lattice; we denote by $L^\vee$ its dual $\Hom_\Z(L,\Z)$. The symmetric bilinear form    on $L$ defines an embedding $L\subset L^\vee$. The  {\em discriminant group}  is the finite abelian group $D(L):=L^\vee/L $; it is endowed with the symmetric bilinear form $b_L:D(L)\times D(L)\to \Q/\Z$ defined by 
$b_L([w],[w']):= w\cdot_\Q w'\pmod{  \Z}$ (\cite{nik}, \S1, 3$^{\rm o}$).  We define the {\em divisibility} $\div(w)$ of a non-zero element $w$ of $L$ as the positive generator of the ideal $w\cdot L\subset \Z$, so that $w/\div(w)$ is primitive in $L^\vee$. We set
  $w_*:=[w/\div(w)]\in D(L)$. If $w$ is primitive,  $\div(w)$ is the order of $w_*$ in $D(L)$.
  
  \begin{prop}
  Let $X$ be a   fourfold of type $\cX_{10}$. The {\em vanishing cohomology lattice}
 $ H^4(X ,\Z)_{\rm van} $
is even and has signature $(20,2)$ and discriminant group    $  (\Z/2\Z)^2$. It is isometric to
  \begin{equation}\label{lam}
\Lambda:= 2E_8\oplus 2U\oplus 2A_1.
\end{equation}
  \end{prop}
  
  \begin{proof} By (\ref{hodx4}), the Hodge structure on $H^4(X)$ has weight 2 and the  unimodular lattice $ \Lambda_X:= H^4(X ,\Z)$, endowed with the intersection form, has signature $(22,2)$. Since $22-2$ is not divisible by 8, this lattice   must be odd, 
  hence of type $22\langle1\rangle \oplus 2\langle-1\rangle $, often denoted by $I_{22,2}$  (\cite{ser}, Chap. V, \S2, cor. 1 of th. 2 and th. 4).

  The intersection form on $ \Lambda_G:=H^4(G(2,V_5),\Z)\vert_X$ has matrix $\begin{pmatrix}2&2\\2&4\end{pmatrix}$ in the basis $(\sigma_{1,1}\vert_X,\sigma_2\vert_X)$.
  It is of type $2\langle1\rangle$ and embeds as a primitive sublattice in $H^4(X ,\Z)$. 
The  vanishing cohomology lattice
$\Lambda_X^0:=H^4(X ,\Z)_{\rm van}:=\Lambda_G^\bot$
therefore has signature $(20,2)$ and    $D(\Lambda_X^0)\isom D(\Lambda_G)\isom (\Z/2\Z)^2$ (\cite{nik}, Proposition 1.6.1).

 An element $x$ of $I_{22,2}$ is {\em characteristic} if
  $$\forall y\in I_{22,2}\quad  x\cdot y\equiv y^2  \pmod{2}. $$ 
 The lattice $x^\bot$ is then even.
One has from \cite{bohi}, \S16.2,
  \begin{equation}\label{bh}
\begin{array}{rcl}
c_1(T_X)&=&2\sigma_1\vert_X,\\
c_2(T_X)&=&4\sigma_1^2\vert_X-\sigma_2\vert_X.
\end{array}
  \end{equation}
 Wu's formula  (\cite{wu}) then gives  
  \begin{equation}\label{wu}
\forall y\in \Lambda_X\quad y^2\equiv y\cdot (c_1^2+c_2)\equiv y\cdot \sigma_2\vert_X  \pmod{2}. 
  \end{equation} 
 In other words, $\sigma_2\vert_X$ is   characteristic, hence $\Lambda_X^0$ is an even lattice. As one can see from Table (15.4) in \cite{cs}, there is only one genus of even lattices with signature $(20,2)$ and discriminant   group  $(\Z/2\Z)^2$ (it is denoted by $II_{20,2}(2_I^2)$ in that table); moreover, there is only one isometry class in that genus (\cite{cs}, Theorem 21). In other words, any lattice with these characteristics, such as the one defined in (\ref{lam}), is isometric to $\Lambda_X^0$. \end{proof}

One can also check that $\Lambda$ is  the orthogonal in $ I_{22,2}$ of the lattice generated by the vectors
\begin{equation*}\label{uv}
u:=e_1+e_2
\quad{\rm and}\quad v':=e_1+\cdots+e_{22}-3f_1-3f_2
\end{equation*}
in the canonical basis $(e_1,\dots,e_{22},f_1,f_2)$ for $I_{22,2}$. Putting everything together, we see that there is an isometry $\gamma:\Lambda_X\isomto I_{22,2}$ such that
\begin{equation}\label{isog}\gamma(\sigma_{1,1}\vert_X)=u,\quad\gamma(\sigma_2\vert_X)=v',\quad \gamma(\Lambda^0_X)\isom\Lambda.
\end{equation}
We let $\Lambda_2\subset I_{22,2}$ be the rank-2 sublattice $\langle u,v'\rangle=\langle u,v\rangle$, where $v:=v'-u$. Then $u$ and $v$ both have divisibility 2, 
 $D(\Lambda_2)=\langle u_*,v_*\rangle$, and the matrix of $b_{\Lambda_2}$ associated with these generators is   $\begin{pmatrix}1/2&0\\0&1/2\end{pmatrix}$.

 \subsection{Lattice automorphisms}\label{s42}
One can   construct $I_{20,2}$  as an overlattice of $\Lambda$ as follows. Let   $e$ and $f$ be respective generators for the last two factors $ A_1$ of $\Lambda$ (see (\ref{lam})). They both have divisibility 2 and $D(\Lambda) \isom (\Z/2\Z)^2$, with generators $e_*$ and $f_*$; the form $b_\Lambda$ has matrix   $\begin{pmatrix}1/2&0\\0&1/2\end{pmatrix}$.  In particular,  $e_*+f_*$ is the only isotropic non-zero element in $D(\Lambda)$. By \cite{nik}, Proposition 1.4.1, this implies that there is a unique unimodular overlattice   of $\Lambda$. Since   there is just one isometry class of unimodular
lattices of signature $(20,2)$, this is  $ I_{20,2}$.

Note that 
 $\Lambda$ is an even sublattice of index 2 of $I_{20,2}$, so it
  is the maximal even sublattice $\{x \in I_{20,2} \mid  x^2\hbox{ even} \}$ (it is contained in that sublattice,
  and it has the same index in $I_{20,2}$).

Every automorphism of $I_{20,2}$ will preserve the maximal  even sublattice, so $O(I_{20,2}) $ is a subgroup of $  O(\Lambda)$. On the other hand, the group $O(
D(\Lambda))$ has order 2 and fixes $e_*+f_*$. It follows that
 every automorphism of $\Lambda$ fixes $I_{20,2}$, and we obtain 
  $O(I_{20,2}) \isom  O(\Lambda)$.

Now let us try to extend  to $I_{22,2}$ an automorphism $ \Id \oplus h $ of $\Lambda_2\oplus \Lambda$. Again, this automorphism  permutes the overlattices of    $\Lambda_2\oplus \Lambda $, such as $I_{22,2}$, according to its action on 
$D(\Lambda_2)\oplus D(\Lambda)$. By \cite{nik}, overlattices correspond to isotropic subgroups of $ D(\Lambda_2)\oplus D(\Lambda)$ that map injectively to both factors. Among them is $I_{22,2}$; after perhaps permuting $e$ and $f$, it corresponds to the (maximal isotropic) subgroup 
$$\{0,u_*+e_*,v_*+f_* , u_*+v_*+e_*+f_*\}.$$
 Any automorphism of $\Lambda$ leaves $e_*+f_*$ fixed. So either $h$ acts trivially on $D(\Lambda)$, in which case $\Id \oplus h $ leaves $I_{22,2}$ fixed hence extends to an automorphism of $I_{22,2}$; or $h$ switches the other two non-zero elements, in which case $ \Id \oplus h  $ does not extend  to $I_{22,2}$.

In other words, the image of the restriction map
 $$  \{g\in O(I_{22,2}) \mid  g\vert_{\Lambda_2}=  \Id  \}\hra O(\Lambda)$$
 is the {\em stable orthogonal group} 
 \begin{equation}\label{sog}
\widetilde O(\Lambda):=\Ker (O(\Lambda) \to O(D(\Lambda)).
\end{equation}
It has index 2 in $O(\Lambda) $ and a generator for the quotient is the involution $r\in O(\Lambda)$ 
 that exchanges $e$ and $f$ and is the identity on $\langle e,f\rangle^\bot$. Let $r_2$ be the involution of $\Lambda_2$ that exchanges $u$ and $v$. It follows from the discussion above that 
the involution $ r_2 \oplus r $ of $\Lambda_2\oplus \Lambda$ extends to an involution $r_I $ of $I_{22,2}$.

\subsection{The   period domain and the period map}\label{sec43}

Fix a lattice $\Lambda$ as in (\ref{lam}); it has signature $(20,2)$. The   manifold
 $$\Omega:=\{\omega\in \P(\Lambda\otimes\C)\mid   (\omega\cdot \omega) = 0\ , (\omega \cdot \bar \omega) < 0\} $$
is a homogeneous space for the real Lie group $\SO(\Lambda\otimes\R)\isom\SO(20,2)$. This group has two components, and one of them reverses the orientation on the negative definite part of $\Lambda \otimes\R$. It follows that $\Omega$ has two components, $\Omega^+$ and $\Omega^-$, 
 both isomorphic to the 20-dimensional open complex manifold $\SO_0(20,2)/\SO(20)\times \SO(2)$, a bounded symmetric domain of type IV.

 Let $\cU $ be a smooth (irreducible) quasi-projective variety parametrizing all  
  fourfolds of type $\cX_{10}$. Let $u$ be a general point of $\cU$ and let $X$ be the corresponding fourfold. 
 The group $\pi_1(\cU,u)$ acts on the lattice $\Lambda_X:=H^4(X,\Z)$ by isometries  and the image    $\Gamma_X$   of the morphism $\pi_1(\cU,u)\to O(\Lambda_X)$ is called the monodromy group.   The group $\Gamma_X$  
 is contained in the subgroup  (see (\ref{sog}))
 $$\widetilde O (\Lambda_X):=\{g\in O(\Lambda_X)\mid g\vert_{\Lambda_G}=\Id  \}.$$
 Choose an isometry $\gamma:\Lambda_X\isomto I_{22,2}$ satisfying (\ref{isog}). It induces an isomorphism $\widetilde O (\Lambda_X)\isom  \widetilde O (\Lambda )$. The group  $\widetilde O (\Lambda )$ acts  on the manifold $\Omega$ defined   above and, by a theorem of Baily and Borel, the quotient $\cD:=   \widetilde  O(\Lambda)\backslash \Omega$ has the structure of an irreducible quasi-projective variety.  One defines as usual a period map $ \cU\to \cD $   by sending a  variety   to its period; it is     an algebraic morphism.
 It descends to ``the'' period map
 $$ \wp: \cX_{10}\to \cD. $$ 
By Theorem \ref{sub}, $\wp$ is dominant with 4-dimensional smooth fibers {\em as a map of stacks.} 

\begin{rema}
As in the three-dimensional case (\cite{dim}), we do not know whether our fourfolds have a  coarse moduli space, even in the category of algebraic spaces. If such a space ${\mathbf X}_{10}$ exists, note however that it is {\em singular along the Gushel locus:}   any fourfold $X$ of Gushel type has a canonical involution; if  $X$ has no other non-trivial
   automorphisms, ${\mathbf X}_{10}$ is then locally around $[X]$  the product of a 22-dimensional germ and the germ of a surface node. The fiber of the period map $   {\mathbf X}_{10}\to \cD $ then  has multiplicity 2   along the surface corresponding to Gushel fourfolds.   \end{rema}

  \section{Special fourfolds}\label{sspe}
  
  Following \cite{has}, \S3, we say that a   fourfold $X$ of type $\cX_{10}$ is {\em special} if it contains a surface whose cohomology class ``does not come'' from $ G(2,V_5)$. Since the Hodge conjecture is true (over $\Q$) for Fano fourfolds (more generally,  by \cite{comu}, for all uniruled fourfolds), this is equivalent to saying that the rank of the (positive definite) lattice $H^{2,2}(X)\cap H^4(X,\Z)$ is at least 3.    The set of special fourfolds is sometimes called the Noether-Lefschetz locus (by Corollary \ref{simple}, a very general $X$ is not special).

    \subsection{Special loci}\label{sl}
   
For each primitive, positive definite, rank-3 sublattice $K\subset  I_{22,2}$   containing   the lattice $\Lambda_2$ defined at the end of \S\ref{s41}, we define an irreducible hypersurface of $\Omega^+$ by setting
$$\Omega_K:=\{\omega\in\Omega^+\mid K\subset \omega^\bot\}.$$
 A fourfold $X$ is {\em special} if and only if its period is in one of these (countably many)  hypersurfaces.  
We now investigate these lattices $K$.

\begin{lemm}\label{le61}
The discriminant $d$ of $K$ is positive and $d\equiv 0,2,{\rm or}\ 4\pmod{8}$.   \end{lemm}

\begin{proof}
Since $K$ is positive definite, $d$ must be positive. Completing the basis $(u,v)$ of $\Lambda_2$ from \S\ref{s41} to a basis of $K$, we see that the matrix of the intersection form in that basis is
 $\begin{pmatrix}2&0&a\\0&2&b\\a&b&c
   \end{pmatrix}$, whose determinant is $d=4c-2(a^2+b^2)$. By Wu's formula (\ref{wu}) (or equivalently, since $v$ is characteristic), we have $c\equiv a+b\pmod{2}$, hence $d\equiv 2(a^2+b^2)\pmod{8}$. This proves the lemma.\end{proof}
   
  We keep the notation of \S\ref{spp}.

   \begin{prop}\label{p53}
    Let $d$ be a positive integer such that $d\equiv 0,2,{\rm or}\ 4\pmod{8}$ and let $\cO_d$ be the set of  orbits for the action  of the group 
  $$\widetilde O(\Lambda ) = \{g\in O(I_{22,2})\mid g\vert_{\Lambda_2}=\Id\}\subset O (\Lambda ) $$
 on the set of primitive, positive definite,  rank-3, discriminant-$d$, sublattices $K\subset I_{22,2}$      containing    $\Lambda_2$. Then,  
 \begin{itemize}
\item[\rm a)] if $d\equiv 0 \pmod{8}$, $\cO_d$ has one element, and $K  \isom \begin{pmatrix}2&0&0\\0&2&0\\
0&0&d/4
   \end{pmatrix} ${\rm ;}
\item[\rm b)]  if $d\equiv 2 \pmod{8}$, $\cO_d$ has two elements, which are interchanged by the involution $r_I$ of $I_{22,2}$, and $K   \isom   \begin{pmatrix}2&0&0\\
0&2&1\\0&1&(d+2)/4
   \end{pmatrix} ${\rm ;}
\item[\rm c)] if $d\equiv 4 \pmod{8}$, $\cO_d$ has one element, and $K   \isom   \begin{pmatrix}2&0&1\\
0&2&1\\1&1&(d+4)/4
   \end{pmatrix} $.
\end{itemize} 
  \end{prop}

In case b), one orbit is characterized by the properties $K\cdot u=\Z$ and $K\cdot v=2\Z$, and the other by 
$K\cdot u=2\Z$ and $K\cdot v= \Z$.

\begin{proof}
By a theorem of Eichler (see, \eg, \cite{ghs}, Lemma 3.5), the $ \widetilde O (\Lambda )$-orbit of a primitive vector $w$ in the even lattice $\Lambda$ is determined by its length $w^2$ and  its class $w_*\in D(\Lambda)$.

If $\div(w)=1$, we have $w_*=0$    and the  orbit is determined by $w^2$. The lattice
$\Lambda_2\oplus \Z w$ is primitive: if $\alpha u+\beta v+\gamma w=m w'$, and if $w\cdot w''=1$, we obtain $\gamma=mw'\cdot w''$, hence $\alpha u+\beta v=m((w'\cdot w'')w-w')$ and $m$ divides $\alpha$, $\beta$, and $\gamma$. Its discriminant is $4w^2\equiv 0\pmod{8}$.

If $\div(w)=2$, we have  $w_*\in\{ e_*,f_*,e_*+f_*\}$. Recall    from \S\ref{s42} that $\frac12(u+e)$, $\frac12(v+f)$, and $\frac12(u+v+e+f)$ are all in $I_{22,2}$. It follows that exactly one of $\frac12(u+w)$, $\frac12(v+w)$, and $\frac12(u+v+w)$  is   in $I_{22,2}$, and $ \Lambda_2\oplus \Z w $ has index 2 in its saturation   $K$ in $I_{22,2}$. In particular, $K$ has discriminant $w^2$. If $w_*\in\{ e_*,f_*\}$, this is $\equiv 2\pmod{8}$; if $w_*=e_*+f_*$, this is $\equiv 4\pmod{8}$.

Now if $K$ is   a lattice as in the statement of the proposition, we let $K^\bot$ be its orthogonal in $ I_{22,2}$, so that the rank-1 lattice $K^0:=K\cap\Lambda$ is the orthogonal of $K^\bot$ in $\Lambda$. From $K^0\subset \Lambda$, we can therefore recover $K^\bot$, then $K\supset \Lambda_2$. The preceding discussion applied to a generator $w$ of $ K^0$  gives the statement, except that we still have to prove that there are indeed elements $w$ of the various types for all $d$, \ie, we need construct elements in each orbit to show they are not empty.

Let $u_1,u_2$ be standard generators for a hyperbolic factor $U$ of $\Lambda$. For any   integer $m$, set $w_m:= u_1+mu_2$.
we have $w_m^2=2m$ and $\div(w_m)=1$. The lattice $\Lambda_2\oplus \Z w_m$ is saturated with  discriminant $8m$.

We have $(e+2w_m)^2=8m+2$ and $\div(e+2w_m)=2$. The saturation  of the lattice $\Lambda_2\oplus \Z (e+2w_m)$ has discriminant $d=8m+2$, and similarly upon replacing $e$ with $f$ (same $d$) or $e+f$ ($d=8m+4$). \end{proof}

Let $K$ be a  lattice   as above. The image in $\cD=   \widetilde  O(\Lambda)\backslash \Omega$ of the hypersurface $\Omega_K\subset  \Omega^+
 $ depends only on the $\widetilde O(\Lambda)$-orbit of $K$. Also, the involution $r\in O(\Lambda) $ induces a non-trivial involution $r_\cD$ of $\cD$.
 
 \begin{coro}
 The periods of the special fourfolds  of discriminant $d$ are contained in
  \begin{itemize}
\item[\rm a)] if $d\equiv 0 \pmod{4}$, an   irreducible hypersurface   $\cD_d\subset \cD$;
\item[\rm b)]  if $d\equiv 2 \pmod{8}$, the union of two irreducible hypersurfaces $\cD'_d$ and $\cD''_d$, which are interchanged by the involution $r_\cD$.\end{itemize} 
 \end{coro}
 
Assume $ d\equiv 2 \pmod{8}$  (case b)). Then, $\cD'_d$ (resp. $\cD''_d$) corresponds to lattices $K$ with $K\cdot u=\Z$ (resp.  $K\cdot v=\Z$). In other words, given a fourfold $X$ of type $\cX_{10}$ whose period point is in $\cD_d=\cD'_d\cup \cD''_d$, it is in $\cD'_d$ if $K\cdot \sigma_1^2 \subset 2\Z$, and it is in $\cD''_d$ if $K\cdot \sigma_{1,1} \subset 2\Z$.

\begin{rema}\label{rem}
Zarhin's  argument, already used in the proof of Corollary \ref{simple}, proves that if $X$ is a fourfold   whose period is very general in any given $\cD_d$, the lattice   $K= H^4(X,\Z) \cap H^{2,2}(X)$ has rank exactly 3 and the rational  Hodge structure $K^\bot\otimes\Q$ is simple.
\end{rema}

\subsection{Associated K3 surface}

As we will see in the next section, K3 surfaces often occur in the geometric description of special fourfolds $X$ of type $\cX_{10}$. This is related to the fact that, for some values of $d$, the   non-special cohomology of $X$ looks like the primitive cohomology of a K3 surface.
      
Following   \cite{has}, we determine, in each case of Proposition \ref{p53},  the discriminant group   of the {\em non-special lattice} $K^\bot $ and the symmetric form $b_{K^\bot}=-b_K$. We then find all cases when   the non-special lattice of $X$ is isomorphic (with a change of sign)  to the primitive cohomology lattice of a (pseudo-polarized, degree-$d$) K3 surface.
 Although this property is only    lattice-theoretic, the surjectivity of the period map for K3 surfaces then produces an actual K3 surface, which is said to be ``associated with $X$.'' For $d=10$, we will see  in \S\ref{splane} and \S\ref{qua}  geometrical constructions of the associated K3 surface.

Finally, there are other cases where geometry provides an ``associated'' K3 surface $S$ (see \S\ref{sno}), but not in the sense  considered here: the Hodge structure of $S$ is only isogeneous to that of the fourfold. So there might be integers $d$ not in the list provided by the proposition below,  for which   special fourfolds of discriminant $d$ are still related to K3 surfaces (of degree different from $d$).

   \begin{prop}\label{p54}
    Let $d$ be a positive integer such that $d\equiv 0,2,{\rm or}\ 4\pmod{8}$ and let $(X,K)$ be a special fourfold of type $\cX_{10}$ with discriminant $d$. 
 Then,  
 \begin{itemize}
\item[\rm a)] if  $d\equiv 0 \pmod{8}$,  we have $  D(K^\bot)   \isom (\Z/2\Z)^2\times (\Z/(d/4)\Z)${\rm ;}
\item[\rm b)]  if $d\equiv 2 \pmod{8}$,  we have $D(K^\bot)   \isom  \Z/d\Z $ and we may choose this isomorphism so that $b_{K^\bot}(1,1)= -\frac{d+8}{2d}\pmod{\Z}${\rm ;}
 \item[\rm c)] if $d\equiv 4 \pmod{8}$,  we have $D(K^\bot)   \isom  \Z/d\Z $ and we may choose this isomorphism so that $b_{K^\bot}(1,1)= -\frac{d+2}{2d}\pmod{\Z}$.
\end{itemize} 
The lattice $K^\bot$ is isomorphic to the  opposite of the  primitive cohomology lattice of a pseudo-polarized K3 surface (necessarily of degree $d$) if and only if we are in case b) or c) and the only odd  primes that divide $d$ are $\equiv 1\pmod{4}$. 

In these cases, there exists a pseudo-polarized, degree-$d$, K3 surface $S$ such that the Hodge structure $H^2(S,\Z)^0(-1)$ is isomorphic to $K^\bot$. Moreover, if the period point of $X$ is not in $\cD_8$, the pseudo-polarization is a polarization.
 \end{prop}

The first values of $d$ that satisfy the conditions for the existence of an associated K3 surface are: 2, 4, 10, 20, 26, 34, 50, 52, 58, 68, 74, 82, 100...

\begin{proof} Since $I_{22,2}$ is unimodular, we have $(D({K^\bot}),b_{K^\bot})\isom (D(K),-b_K)$ (\cite{nik}, Proposition 1.6.1). Case a) follows from Proposition \ref{p53}.

Let $e$, $f$, and $g$ be the generators of $K$ corresponding to the matrix given in Proposition \ref{p53}. The matrix of $b_{K^\bot}$ in the dual basis $(e^\vee,f^\vee,g^\vee)$ of $K^\bot$
 is the inverse of that matrix.

In case b), one checks that $e^\vee+g^\vee $ generates $D(K)$, which is isomorphic to $\Z/d\Z$. Its square is $\frac12+\frac4d= \frac{d+8}{2d}$.

In case c), one
  checks that $e^\vee $ generates $D(K)$, which is isomorphic to $\Z/d\Z$. Its square is $   \frac{d+2}{2d}$.
  
  The opposite of the primitive cohomology lattice of a pseudo-polarized K3 surface  of degree $d$ has discriminant group $\Z/d\Z$ and the square of a generator is $\frac1d$. So case a) is impossible. 
    
  In case b), the   forms are conjugate if and only if $-\frac{d+8}{2d}\equiv \frac{n^2}{d} \pmod{\Z}$ for some   integer  $n$ prime to $d$, or $-\frac{d+8}{2}\equiv n^2  \pmod{d}$. Set $d=2d'$ (so that $d'\equiv 1\pmod{4}$); then this is equivalent to saying that $d'-4$ is a square  in the ring $\Z/d\Z$. Since $d'$ is odd, this ring is isomorphic to $\Z/2\Z\times \Z/d'\Z$, hence this is equivalent to asking that $-4$, or equivalently $-1$, is a square in $\Z/d'\Z$. This   happens if and only if the only odd primes that divide $d'$ (or $d$) are $\equiv 1\pmod{4}$.  
  
  In case c), the reasoning is similar: we need
 $-\frac{d+2}{2d}\equiv \frac{n^2}{d} \pmod{\Z}$ for some integer $n$ prime to $d$.  Set $d=4d'$, with $d'$ odd. This is equivalent to
  $-2\equiv   2n^2  \pmod{d'}$, and we conclude as above.
  
  As already explained, the existence of the polarized K3 surface $(S,f)$ follows from the surjectivity of the period map for K3 surfaces. Finally, if $\wp([X])$ is not in $\cD_8$, there are no classes of type $(2,2)$ with square 2 in 
 $H^4(X,Z)_{\rm van}$, hence no $(-2)$-curves on $S$ orthogonal to $f$, so $f$ is  
a polarization.
 \end{proof}

\subsection{Associated cubic fourfold}

Cubic fourfolds also sometimes occur in  the geometric description of special fourfolds $X$ of type $\cX_{10}$ (see \S\ref{rplane}). We determine   for which values of $d$   the   non-special cohomology of $X$  is isomorphic the non-special cohomology of a  special  cubic fourfold. Again, this is only a lattice-theoretic association, but the surjectivity of the period map for cubic   fourfolds then produces a (possibly singular) actual cubic.  We will see in \S\ref{rplane}   that   some  special fourfolds $X$ of discriminant 12 are actually birationally isomorphic to their associated  special cubic  fourfold.

      \begin{prop}\label{p55}
    Let $d$ be a positive integer such that $d\equiv 0,2,{\rm or}\ 4\pmod{8}$ and let $(X,K)$ be a special fourfold of type $\cX_{10}$ with discriminant $d$.
The lattice $K^\bot$ is isomorphic to the non-special    cohomology lattice of a (possibly singular) special cubic fourfold (necessarily of discriminant $d$) if and only if 
 \begin{itemize}
\item[\rm a)] either  $d\equiv 2\ {\it or}\ 20 \pmod{24}$,  and the only odd  primes that divide $d$ are $\equiv \pm1\pmod{12}${\rm ;}
\item[\rm b)]  or  $d\equiv 12\ {\it or}\ 66 \pmod{72}$,  and the only    primes $\ge 5$ that divide $d$ are $\equiv \pm1\pmod{12}$.
\end{itemize} 

In these cases, if moreover the period point of $X$ is general in $\cD_d$ and $d\ne 2$,   there exists a {\em smooth} special cubic fourfold   whose non-special Hodge structure   is isomorphic to $K^\bot$.  
 \end{prop}

The first values of $d$ that satisfy the conditions for the existence of an associated cubic fourfold are: 2, 12,   26,  44, 
66,  74, 92, 122, 138, 146, 156, 194...

\begin{proof}Recall from \cite{has}, \S4.3, that  (possibly singular) special cubic fourfolds of positive discriminant $d $ exist for $d\equiv 0\ {\rm or}\ 2\pmod{6}$  (for $d=2$, the associated cubic fourfold is the  (singular) determinantal cubic; for $d=6$, it is nodal). Combining that condition with that of Lemma \ref{le61}, we obtain the necessary condition  $d\equiv 0,2,8, 12,18,20\pmod{24}$. Write
  $d=24d'+e$, with $e\in\{0,2,8, 12,18,20\}$.

 Then, one needs  to check whether the discriminant forms are isomorphic. Recall from \cite{has}, Proposition 3.2.5, that the discriminant group of the non-special lattice of a special cubic fourfold  of discriminant $d$ is isomorphic to 
 $(\Z/3\Z)\times(\Z/(d/3)\Z)$ if $d\equiv 0\pmod6$, and to $\Z/d\Z$ if $d\equiv 2\pmod6$. This excludes $e=0$ or 8; for $e=12$, we need $d'\not\equiv 1\pmod3$, and for $e=18$, we need $d'\not\equiv 0\pmod3$. In all these cases, the discriminant group is cyclic.
 
 When $e=2$,   the discriminant forms are conjugate if and only if $-\frac{d+8}{2d}\equiv n^2\frac{2d-1}{3d} \pmod{\Z}$ for some integer  $n$ prime to $d$ (Proposition \ref{p54} and \cite{has}, Proposition 3.2.5), or equivalently, since 3 is invertible modulo $d$, if and only if $ \frac{d}{2}+12\equiv 3\frac{d+8}{2}\equiv n^2  \pmod{d}$. This is equivalent to saying that $12d'+13$ is a square in $\Z/d\Z\isom (\Z/(12d'+1)\Z)\times(\Z/2\Z)$, or that 3 is a square in $\Z/(12d'+1)\Z$.
Using quadratic reciprocity, we see that this is equivalent to saying that the only odd primes that divide $d$ are $ \equiv \pm1\pmod{12}$.

When $e=20$, we need $-\frac{d+2}{2d}\equiv n^2\frac{2d-1}{3d} \pmod{\Z}$ for some integer  $n$ prime to $d$, 
or equivalently,   $ \frac{d}{2}+3\equiv  n^2  \pmod{d}$. Again, we get the same condition.

 When $e=12$, we need $9\!\!\not|\, d$ and $-\frac{d+2}{2d}\equiv n^2\left( \frac23-\frac{3}{ d}\right) \pmod{\Z}$ for some integer  $n$ prime to $d$, or equivalently $ -12d'-7 \equiv n^2(16d'+5)  \pmod{d}$. Modulo 3, we get that $1-d'$ must be a non-zero square, hence $3\mid d'$. Modulo 4, there are no conditions. Then we need  
 $  1  \equiv 3n^2   \pmod{2d'+1}$ and we conclude as above.
 
 Finally, when $e=18$, we need $9\!\!\not|\, d$ and $-\frac{d+8}{2d}\equiv n^2\left( \frac23-\frac{3}{ d}\right) \pmod{\Z}$ for some   $n$ prime to $d$, or equivalently $ -12d'-13 \equiv n^2(16d'+9)  \pmod{d}$. Modulo 3, we get   $d'\equiv 2\pmod3$, and then  $  4 \equiv 3n^2   \pmod{4d'+3}$ and we conclude as above.
 
At this point, we have a Hodge structure on $K^\bot$ which is, as a lattice, isomorphic to the non-special cohomology of a special cubic fourfold. It corresponds to a point in the period domain $\cC$ of cubic fourfolds. To make sure that it corresponds to a (then unique) smooth cubic fourfold, we need to check that it is not in the special loci $\cC_2\cup\cC_6$  (\cite{laza}, Theorem 1.1). If the period point of $X$ is general in $\cD_d$, the period point in $\cC$ is general in $\cC_d$, hence is not in $\cC_2\cup\cC_6$ if $d\notin\{2,6\}$.  \end{proof}

  \begin{rema}One can be more precise and   figure out explicit conditions on $\wp([X])$ for the associated cubic fourfold to be smooth,  but calculations are complicated: if $d=6e$ and we are in $\cC_6$, there is a   class $v$ with $v^2=2$ and $v\cdot h=0$, so we get a rank-3 lattice of $(2,2)$-classes with intersection matrix $\begin{pmatrix}3&0&0 \\ 0&2e&a\\0&a&2\end{pmatrix}$ (with $a^2<4e$) and the $(2,2)$-class $au-2ev$ is non-special, hence corresponds to a non-special class in $X$ with square $2e(4e-a^2)$; if we are in $\cC_2$, there is a   class $v$ with $v^2=6$ and $v\cdot h=0$, and we proceed similarly.    For example, when $d=12$, we find that it is enough to assume  $\wp([X])\notin\cD_2\cup\cD_4\cup \cD_8\cup \cD_{16}\cup \cD_{28}\cup \cD_{60}\cup\cD_{112}\cup\cD_{240}$.
 \end{rema} 

   \section{Examples of special   fourfolds}\label{exam}

   Assume that a fourfold $X$ of type $\cX_{10}$  contains a smooth surface $S$. Then, by (\ref{bh}),
   $$c (T_X)\vert_S= 1+2\sigma_1 \vert_S+ (4\sigma_1^2\vert_S-\sigma_2\vert_S)=c(T_S)c(N_{S/X}).
   $$
   This implies $c_1(T_S)+c_1(N_{S/X})=2\sigma_1 \vert_S$ and
   $$4\sigma_1^2\vert_S-\sigma_2\vert_S= c_1(T_S) c_1(N_{S/X})+c_2(T_S)+c_2(N_{S/X}).
   $$
   We obtain
$$
(S)_X^2=c_2(N_{S/X})=4\sigma_1^2\vert_S-\sigma_2\vert_S-c_1(T_S) (2\sigma_1 \vert_S-c_1(T_S))-c_2(T_S).
$$
Write  $[S]=a\sigma_{3,1}+b\sigma_{2,2}$ in $G(2,V_5)$. Using Noether's formula,
we obtain
\begin{equation}\label{S2}
(S)_X^2=3a+4b+2K_S\cdot \sigma_1\vert_S+2K^2_S- 12\chi(\cO_S).
\end{equation}
The determinant of the intersection matrix  in the basis $(\sigma_{1,1}\vert_X, \sigma_2\vert_X-\sigma_{1,1}\vert_X, [S])$ is then
\begin{equation}\label{S3}
d=4(S)^2_X-2(b^2+(a-b)^2).
\end{equation}
We remark that $\sigma_2\vert_X-\sigma_{1,1}\vert_X$ is the class of the unique $\sigma$-quadric surface $\Sigma_0$ contained in $X$  (see \S\ref{suni}).

  \subsection{Fourfolds containing a $\sigma$-plane (divisor $\cD''_{10}$)}\label{splane} 
  A $\sigma$-plane  is a 2-plane in $G(2,V_5)$  of the form $\P(V_1\wedge V_4)$; its class in $G(2,V_5)$ is $\sigma_{3,1}$. 
Fourfolds of type $\cX_{10}$ containing such a 2-plane were  already  studied by  Roth (\cite{rot}, \S4) and Prokhorov (\cite{pro}, \S3).
  
     \begin{prop}\label{prop71}
  Inside  $ \cX_{10}$, the  family $\cX_{\sigma\hbox{\tiny\rm -plane}}$ of fourfolds containing  a $\sigma$-plane is irreducible of codimension 2. The period map induces a dominant map $\cX_{\sigma\hbox{\tiny\rm -plane}}\to \cD''_{10}$ whose general fiber has dimension 3 and  is rationally  dominated by a $\P^1$-bundle over a degree-10 K3 surface.
  
  A general member of $\cX_{\sigma\hbox{\tiny\rm -plane}}$ is rational.
  \end{prop}
  
    During the proof, we present an explicit geometrical construction of a general member $X$ of $\cX_{\sigma\hbox{\tiny\rm -plane}}$, starting  from a general degree-10 K3 surface $S\subset \P^6$,  a general point $p$ on $S$, and a smooth quadric $Y$ containing the projection $\widetilde S\subset \P^5$ from $p$.
     The birational isomorphism $Y\dra X$ is given by the linear system of cubics containing  $\widetilde S$.   
 
  \begin{proof}
  A parameter count (\cite{im}, Lemma 3.6) shows that $\cX_{\sigma\hbox{\tiny -plane}}$ is irreducible of codimension 2 in $\cX_{10} $. Let $P\subset X$ be a $\sigma$-plane. From (\ref{S2}), we obtain $(P)^2_X=3$, and from (\ref{S3}),  $d=10$. Since $\sigma_1^2\cdot P$ is odd, we are in $\cD''_{10}$.
   
For $X$ general in $\cX_{\sigma\hbox{\tiny -plane}}$ (see \cite{pro}, \S3, for the precise condition), the image of the  projection $\pi_P:X\dra \P^5$ from $P$ is a smooth quadric $Y\subset \P^5$   and, if $\widetilde X\to X$ is the blow-up of $P$, the projection $\pi_P$ induces a birational morphism $ \widetilde X\to Y$ which is the blow-up of a smooth degree-9 surface $\widetilde S$, itself  the blow-up of a smooth degree-10 K3 surface $S$  at one point (\cite{pro}, Proposition 2).

 Conversely, starting from a (general) degree-10 K3 surface $S\subset \P^6$, project it from a point on $S$ to obtain an embedding $\widetilde S\subset \P^5$. The surface $\widetilde S$ is then contained in a pencil of quadrics. For each such smooth quadric, one can reverse the construction above and produce a fourfold $X$ containing a $\sigma$-plane (we will go back in more details to this construction during the proof of Theorem \ref{th81}).

There are isomorphisms
 of polarized integral Hodge structures  
\begin{eqnarray*}
H^4(\widetilde X,\Z) &\isom& H^4(X,\Z)\oplus H^2(P,\Z)(-1)\\
&\isom& H^4(Y,\Z)\oplus H^2(\widetilde S,\Z)(-1)\\
&\isom& H^4(Y,\Z)\oplus H^2( S,\Z)(-1)\oplus \Z(-2).
\end{eqnarray*}
 For $S$ very general, the Hodge structure $H^2( S,\Q)_0 $ is simple, hence it is isomorphic to the non-special cohomology $K^\bot\otimes\Q$ (where $K$ is the  lattice spanned by $H^4(G(2,V_5),\Z)$ and $[P]$ in $H^4(X,\Z)$). Moreover, the lattice
 $H^2( S,\Z)_0(-1)$ embeds isometrically  into $K^\bot$. Since they both have rank 21 and discriminant 10, they are isomorphic. The surface
 $S$ is thus the (polarized) K3 surface associated with $X$ as in  Proposition \ref{p54}. 
 
 Since the period map for polarized degree-10 K3 surfaces is dominant onto their period domain, the period map for $\cX_{\sigma\hbox{\tiny\rm -plane}}$ is dominant onto $\cD''_{10}$ as well. Since the Torelli theorem for K3 surfaces holds, $S$ is determined by the period point of $X$, hence the fiber $\wp^{-1}([X])$ is rationally dominated the family of pairs $(p,Y)$, where $p\in S$ and $Y$ belongs to the pencil of quadrics in $\P^5$ containing $\widetilde S$.
 \end{proof}

With the notation above,  the inverse image  of the quadric $Y\subset\P^5$  by the projection $\P^8\dra \P^5$ from $P$ is a rank-6 non-Pl\"ucker quadric in $\P^8$ containing $X$, with vertex $P$. We will show in \S\ref{qdp}  that $\cX_{\sigma\hbox{\tiny -plane}}$ is contained in the    irreducible hypersurface of
  $\cX_{10}$  parametrizing the  fourfolds $X$ contained in such a quadric.

 \subsection{Fourfolds containing a $\rho$-plane (divisor $\cD_{12}$)}\label{rplane} A $\rho$-plane   is a 2-plane  in $ G(2,V_5)$  of the form $\P(\wedge^2V_3)$; its class in $G(2,V_5)$ is $\sigma_{2,2}$. Fourfolds of type $\cX_{10}$ containing such a 2-plane were  already  studied by  Roth (\cite{rot}, \S4).
 
   \begin{prop}\label{pd12}
  Inside $ \cX_{10}$, the  family $\cX_{\rho\hbox{\tiny\rm -plane}}$ of fourfolds containing  a $\rho$-plane is irreducible of codimension 3. The period map induces a dominant map $\cX_{\rho\hbox{\tiny\rm -plane}}\to \cD_{12}$ whose general fiber is   the   union of two rational surfaces.
  
   A general member of $\cX_{\rho\hbox{\tiny\rm -plane}}$ is birationally isomorphic to a cubic fourfold containing a smooth cubic surface scroll.
  \end{prop}
  
    The proof presents a geometrical construction of a general member of $\cX_{\rho\hbox{\tiny\rm -plane}}$,  starting from any smooth cubic fourfold $Y\subset \P^5$ containing a smooth cubic surface scroll  $T$.
        The birational isomorphism $Y\dra X$ is given by the linear system of quadrics containing $T$.

  \begin{proof}
    A parameter count (\cite{im}, Lemma 3.6) shows that $\cX_{\rho\hbox{\tiny -plane}}$ is irreducible of  codimension 3 in $\cX_{10} $. Let $P=\P(\wedge^2V_3)\subset X$ be a $\rho$-plane. 
    From (\ref{S2}), we obtain $(P)^2_X=4$. From (\ref{S3}), we obtain $d=12$  and  we are in $\cD_{12}$.

As shown in \cite{rot}, \S4,  the image  of the projection $\pi_P:X\dra \P^5$ from $P$ is a cubic hypersurface $Y$ and   the image of the intersection of $X$ with the Schubert hypersurface 
$$\Sigma_P=\{  V_2\subset V_5\mid V_2\cap V_3\ne 0\}\subset G(2,V_5)$$
is a   cubic surface scroll $T$ (contained in $Y$). If $\widetilde X\to X$ is the blow-up of $P$, with exceptional divisor $E_P$, the projection $\pi_P$ induces a birational morphism $\widetilde \pi_P: \widetilde X\to Y$. One checks (with the same arguments as in \cite{pro}, \S3) that all fibers have dimension $\le 1$ hence that $\widetilde \pi_P$ is the blow-up of the smooth surface $T$. The image $ \widetilde \pi_P(E_P)$ is the (singular) hyperplane section   $Y_0:= Y\cap\langle T\rangle$.

Conversely,  a general cubic fourfold $Y$ containing a smooth cubic scroll  contains two families (each parametrized by $\P^2$) of such surfaces (see \cite{hast} and \cite{hast2}, Example 7.12). For each such smooth cubic scroll, one can reverse the construction above and produce a smooth fourfold $X$ containing a $\rho$-plane.

 As in \S\ref{splane},   there are isomorphisms
 of polarized integral Hodge structures  
$$H^4(\widetilde X,\Z) \isom H^4(X,\Z)\oplus H^2(P,\Z)(-1) \isom H^4(Y,\Z)\oplus H^2(T,\Z)(-1).$$ 
Let   $K$ be the  lattice spanned by $H^4(G(2,V_5),\Z)$ and $[P]$ in $H^4(X,\Z)$. For $X$ very general in $\cX_{\rho\hbox{\tiny\rm -plane}}$, the Hodge structure 
$K^\bot\otimes\Q$ is simple (Remark \ref{rem}), hence it is isomorphic to the Hodge structure $ \langle h^2,[T]\rangle^\bot \subset H^4(Y,\Q) $. Moreover, the lattices
 $ K^\bot $ and $ \langle h^2,[T]\rangle^\bot\subset H^4(Y,\Z) $, which both have   rank 21 and   discriminant 12 (see \cite{has}, \S4.1.1), are isomorphic. 
 This case fits into the setting of Proposition \ref{p55}: the special cubic fourfold $Y$ is associated with $X$. 
 
 Finally, since the period map for cubic fourfolds containing a cubic scroll surface is dominant onto the corresponding hypersurface in their period domain, the period map for $\cX_{\rho\hbox{\tiny\rm -plane}}$ is dominant onto $\cD_{12}$ as well. Since the Torelli theorem holds for cubic fourfolds (\cite{voi}), $Y$ is determined by the period point of $X$, hence the fiber $\wp^{-1}([X])$ is rationally dominated the family of smooth cubic scrolls contained in $Y$. It is therefore the union of two rational surfaces.
 \end{proof}

  With the notation above, let $V_4\subset V_5$ be a general hyperplane containing $V_3$. Then $G(2,V_4)\cap X$ is the union of $P$ and a cubic scroll surface.

   \begin{rema}\label{planf}
  {\em No   fourfold $X$ of type $\cX_{10}$ contains infinitely many 2-planes.}
Indeed, if there is a one-dimensional family of 2-planes contained in $X$, two general such 2-planes $P$ and $P'$ will be of the same type (\ie, both $\sigma$-planes, or both $\rho$-planes), and $P\cdot P'=P^2=3$ (for $\sigma$-planes, as seen in the proof of Proposition \ref{prop71}) or 4 (for $\rho$-planes, as seen in the proof of Proposition \ref{pd12}). But one checks that these two distinct 2-planes  meet  in either 0 or 1 points, except if they are $\sigma$-planes of   types $\P(V_1\wedge V_4)$ and $\P(V_1\wedge V'_4)$. For such a 2-plane $\P(V_1\wedge V_4)$  to be contained in $X$, we need $V_4\subset V_1^\bot$; if there are more than one, $V_1$ must be the kernel ${V_1^\omega}$ of the form $\omega$ and these 2-planes must be contained in the surface $\Sigma_0$ defined in \S\ref{suni}. So   this is a contradiction.
  \end{rema}

     \subsection{Fourfolds containing a $\tau$-quadric surface (divisor $\cD'_{10}$)}\label{qua}
A $\tau$-quadric surface in $ G(2,V_5)$  is a linear section of $G(2,V_4)$ (its general hyperplane sections   are $\tau $-conics); its class  in $G(2,V_5)$ is $\sigma_1^2\cdot \sigma_{1,1}=\sigma_{3,1}+\sigma_{2,2}$.

  \begin{prop}
  The closure $\overline \cX_{\tau\hbox{\tiny\rm -quadric}}\subset \cX_{10}$
  of the
    family  of fourfolds containing  a $\tau$-quadric surface is an irreducible component of $\wp^{-1}( \cD'_{10})$.
   The period map induces a dominant map $\cX_{\tau\hbox{\tiny\rm -quadric}}\to \cD'_{10}$ whose general fiber is birationally isomorphic to the quotient by an involution of the symmetric square of a degree-10 K3 surface.
   
    A general member of $\cX_{\tau\hbox{\tiny\rm -quadric}}$ is rational.
  \end{prop}
  
    During the proof, we present a geometrical construction of a general member of $\cX_{\tau\hbox{\tiny\rm -quadric}}$,  starting from a general degree-10 K3 surface $S\subset \P^6$ and two general points  on $S$: if $S_0\subset \P^4$ is the (singular) projection of $S$ from these two points,     the birational isomorphism $\P^4\dra X$ is given by the linear system of quartics containing $S_0$.

\begin{proof}A parameter count shows that $\cX_{\tau\hbox{\tiny -quadric}}$ is irreducible of   codimension 1 in $\cX_{10}$ (one can also use the parameter count at the end of the proof).   Let $\Sigma\subset X$ be a smooth $\tau$-quadric surface.
From (\ref{S2}), we obtain $(\Sigma)^2_X=3$, and from (\ref{S3}), $d=10$. Since $\sigma_1^2\cdot \Sigma$ is even, we are in $\cD'_{10}$.
   The family  $\cX_{\tau\hbox{\tiny -quadric}}$ is therefore a component of the divisor $\wp^{-1}(\cD'_{10})$.
   
   The projection from the 3-plane $\langle \Sigma\rangle$ induces a birational map $X\dra \P^4$ (in particular, $X$ is rational!). If $\eps:\widetilde X\to X$ is the blow-up of $\Sigma $, one checks that it induces a birational {\em morphism} $\pi: \widetilde X\to \P^4$ which is more complicated than just the blow-up of a smooth surface (compare with \S\ref{splane}).
   
    Indeed, since $\Sigma$ is contained in a $G(2,V_4)$,  the quartic surface $X\cap G(2,V_4)$  is   the union of $\Sigma$ and another $\tau$-quadric surface $\Sigma^\star$. The two 3-planes $\langle \Sigma\rangle$  and $\langle \Sigma^\star\rangle$ meet along a 2-plane, hence (the strict transform of) $\Sigma^\star$ is contracted by $\pi$ to a point. Generically, the only quadric surfaces contained in $X$ are the $\sigma$-quadric surface $\Sigma_0$ (defined in \S\ref{suni}) and the $\tau$-quadric surfaces $\Sigma$ and   $\Sigma^\star$. Using the fact that $X$ is an intersection of quadrics, one checks that  $\Sigma^\star$ is the only surface contracted (to a point) by $\pi$.
    
Let $\ell'\subset \widetilde X$ be a line contracted by   $\eps$. If $\ell\subset \widetilde X$ is (the strict transform of) a line contained in $\Sigma^\star$, it meets $\Sigma$ and is contracted by $\pi$. Since $\widetilde X$ has Picard number 2, the   rays $\R^+[\ell]$ and $\R^+[\ell']$ are extremal hence span the cone of curves of $\widetilde X$. These two classes have $(-K_{\widetilde X})$-degree 1, hence  $\widetilde X$ is a Fano fourfold. 
Extremal contractions on smooth fourfolds have been classified   (\cite{am}, Theorem 4.1.3); in our case, we have:
\begin{itemize}
\item   $\pi$ is   a divisorial contraction,   its (irreducible) exceptional divisor $D$ contains  $\Sigma ^\star$, and  $D\lin 3H-4E$;\footnote{Use   $H^4=10$, $ H^3E=0$, $H^2E^2=-2$, $HE^3=0$, $ E^4=3$.} 
\item $S_0:= \pi(D)$ is a surface with a single singular point  $s:=\pi(\Sigma^\star)$, where it is locally the union of two smooth 2-dimensional germs meeting transversely;
\item outside of $s$, the map $\pi$ is the blow-up of $S_0$ in $\P^4$.
\end{itemize}
Let $\widehat X\to \widetilde X$ be the blow-up of $\Sigma^\star$, with exceptional divisor $\widehat E$, and let $\widehat \P^4\to \P^4$ be the blow-up of $s$, with exceptional divisor $\P^3_s$. The strict transform $\widehat S_0\subset \widehat \P^4$ of $S_0$ is the blow-up of its (smooth) normalization $S_0'$ at the two points lying over $s$ and meets $ \P^3_s$ along the disjoint union of the two exceptional curves $L_1$ and $L_2$. There is an induced morphism $\widehat X\to \widehat \P^4$ which is an extremal contraction (\cite{am}, Theorem 4.1.3) hence is the blow-up of the smooth surface $\widehat S_0$, with exceptional divisor the strict transform $\widehat D\subset \widehat X$ of $D$; it induces by restriction a morphism $\widehat E\to  \P^3_s$ which is the blow-up of $L_1\sqcup L_2$.

It follows that we have isomorphisms  of polarized Hodge structures  
\begin{eqnarray*}
H^4(\widehat X,\Z) &\isom& H^4(X,\Z)\oplus H^2(\Sigma,\Z)(-1)\oplus H^2(\Sigma^\star,\Z)(-1)\\
& \isom& H^4(  \P^4,\Z)\oplus H^2(  \P^3_s,\Z)(-1) \oplus \Z[L_1]\oplus  \Z[L_2]\oplus  H^2( S'_0,\Z)(-1).
\end{eqnarray*}
In particular, we have $b_2(S_0')=24+2+2-1-1-1-1=24$ and $h^{2,0}(S'_0)=h^{3,1}(\widehat X)=1$; moreover, the Picard number of $S'_0$ is 3 for $X$ general. The situation is as follows:
$$
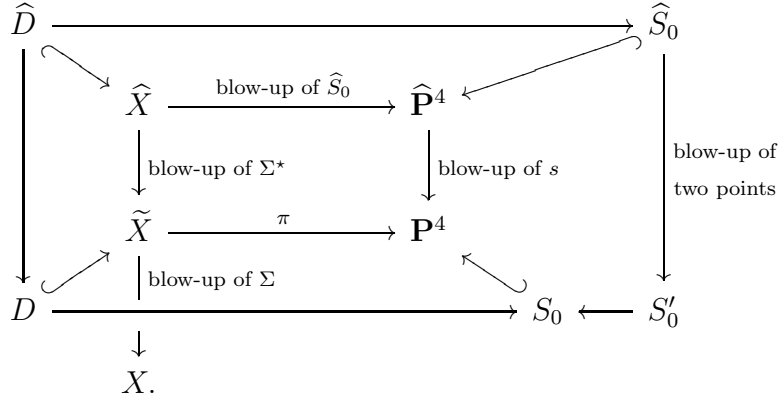
\xymatrix
@C=20pt@M=6pt@R=6pt
{
\widehat D\ar@{_(->}[dr]\ar[rrrrrr]\ar[dddd] &&&&&&\widehat S_0\ar[dddd]^{\begin{matrix}\hbox{\tiny blow-up of}\\\hbox{\tiny two points}\end{matrix}}\ar@{^(->}[dll]\\
&\widehat X \ar[rrr]^{\hbox{\tiny blow-up of $\widehat S_0$}}\ar[dd]^{\hbox{\tiny blow-up of $\Sigma^\star$}}&&&\widehat \P^4\ar[dd]^{\hbox{\tiny blow-up  of $s$}}\\\\
&\widetilde X\ar[rrr]^\pi\ar '[d]^(.6){\hbox{\tiny blow-up of $\Sigma$}} [dd]&&& \P^4\\
D\ar@{^(->}[ur] \ar[rrrrr]&&&&&S_0\ar@{_(->}[ul]&S'_0\ar[l]\\
&X .}
$$
To compute the degree $d$ of $S_0$, we consider the (smooth) inverse image $P\subset \widetilde X$ of a 2-plane in $\P^4$. It is isomorphic to the blow-up of $\P^2$ at $d$ points, hence $K^2_P=9-d$. On the other hand, we have by adjunction
$$K_P\lin (K_{\widetilde X}+2(H-E))\vert_P\lin (-2H+E+2(H-E))\vert_P= -E\vert_P,$$
hence
$K_P^2=E^2(H-E)^2=1$ and $d=8$.

Consider now a general hyperplane $h\subset \P^4$. Its intersection with $S_0$ is a smooth connected curve $C$ of degree 8, and its inverse image in $\widetilde X$ is the blow-up of $h$ along $C$, with exceptional divisor its intersection with $D$. From \cite{ip}, Lemma 2.2.14, we obtain
$$D^3\cdot (H-E)=-2g(C)+2+K_h\cdot C=-2g(C)+2-4\deg(C)=-2g(C)-30,
$$
from which we get $g(C)=6$. In particular, $c_1(S'_0)\cdot h=2$. On the other hand, using a variant of the formula for smooth surfaces in $\P^4$, we obtain
$$d^2-2=10d+c_1^2(S'_0)-c_2(S'_0)+5c_1(S'_0)\cdot h ,
$$
hence $c_1^2(S'_0)-c_2(S'_0)=-28$. 
We can also use the formula from \cite{pro}, Lemma 2:

\begin{eqnarray*}
\widehat D^4&=&(c_2(\widehat \P^4)-c_1^2(\widehat \P^4))\cdot \widehat S_0+c_1(\widehat \P^4)\vert_{\widehat S_0}\cdot c_1(\widehat S_0)-c_2(\widehat S_0)\\
&=&(-15h^2-7[\P^3_s]^2)\cdot \widehat S_0+(-5h^2+3[\P^3_s])\vert_{\widehat S_0}\cdot c_1(\widehat S_0)-c_2(\widehat S_0)\\
&=&(-15h^2-7[\P^3_s]^2)\cdot \widehat S_0+(-5h^2+3[\P^3_s])\vert_{\widehat S_0}\cdot c_1(\widehat S_0)-c_2(\widehat S_0)\\
&=&-120+14-10-6-c_2(\widehat S_0).\end{eqnarray*}
Since $\widehat D^4=D^4=(3H-4E)^4= -150$, we obtain $c_2(\widehat S_0)= 28$, hence $c_2( S'_0)= 26$ and  $c_1^2(S'_0)=-2$. Noether's formula implies  $\chi(S'_0,\cO_{S'_0})=2$, hence $h^1(S'_0,\cO_{S'_0})=0$. The classification of surfaces implies that $S'_0$ is the blow-up at two points of a K3 surface $S$ of degree 10. By the simplicity argument used before, the integral polarized Hodge structures  $H^2( S,\Z)_0(-1)$ and $K^\bot$ are isomorphic:
$S$ is the    (polarized) K3 surface associated with $X$ via  Proposition \ref{p54}.

What happens if we start from the $\tau$-quadric $\Sigma^\star$ instead of $\Sigma$? Blowing-up $\Sigma$ and then the strict transform of $\Sigma^\star$ is not the same as doing it in the reverse order, but the 
end products have a common open subset $\widetilde X^0$ (whose complements have codimension 2). The morphisms $\widetilde X^0\to \widetilde\P^4
 \to \P^3$ (where the second morphism is induced by projection from $s$) are then the same, because they are induced by the projection of $X$ from the 4-plane $\langle \Sigma,\Sigma^\star\rangle$, and the locus where they are not smooth is
the common projection $S_1$ in $\P^3$ of the surfaces $S_0\subset \P^4$ and $S_0^\star\subset \P^4$ from their singular points.

This surface $S_1$ is also the projection of the K3 surface $S\subset \P^6$ from  the 2-plane spanned by  $p,p',q,q'$. The end result is therefore the same K3 surface $S$ (as it should, because its period is determined by that of $X$), but the pair of points is now $q,q'$. We let 
 $\iota_S$ denote the birational involution on $S^{[2]}$ defined by $p+p'\mapsto q+q'$ (in \cite{ogr}, Proposition 5.20, O'Grady proves that for $S$ general, the involution $\iota_S$ is biregular on the complement of  a 2-plane.).

Conversely, let $S=G(2,V_5)\cap Q'\cap \P^6$ be  a general K3 surface of degree 10 and let $p$ (corresponding to $V_2\subset V_5$) and $p'$ (corresponding to $V'_2\subset V_5$)  be two general points on $S$. If $V_4:=V_2\oplus V'_2$, the intersection $S\cap G(2,V_4)$ is a set of four points $p,p',q,q'$ in the 2-plane $\P(\wedge^2V_4)\cap\P^6$. Projecting $S$ from the line $pp'$ gives a non-normal degree-8 surface $S_0:=S_{pp'}\subset \P^4$, where $q$ and $q'$ have been identified. Its normalization $S'_0$ is the blow-up of $S$ at  $p$ and $p'$.
Now let $ \widehat \P^4\to \P^4$ be the blow-up of the singular point of $S_0$, and let $\widehat X\to \widehat \P^4$ be the blow-up of the strict transform of $S_0$ in $\widehat \P^4$. The strict transform in $\widehat X$ of the exceptional divisor $\P^3_s\subset  \widehat \P^4$ can be blown down by $\widehat X\to \widetilde X$.

The resulting smooth fourfold $\widetilde X$ is a Fano variety with Picard number 2. One extremal contraction is $\pi:\widetilde X\to\P^4$. The other extremal contraction gives the desired $X$. This construction depends on 
23 parameters (19 for the surface $S$ and 4 for  $p,p'\in S$).

All this implies (as in the proofs of Propositions \ref{prop71} and   \ref{pd12}) that the period  map for $\cX_{\rho\hbox{\tiny\rm -plane}}$ is dominant onto $\cD'_{10}$, with  fiber birationally isomorphic to  $S^{[2]}/\iota_S$.
 \end{proof}

   \begin{rema}
  {\em No   fourfold $X$ of type $\cX_{10}$ contains infinitely many quadric surfaces.} The proof is   more involved than in Remark \ref{planf}, but still elementary. Since $X$ contains a unique $\sigma$-quadric, we may assume (using Remark \ref{planf}) that $X$ contains   irreducible $\tau$-quadrics $\Sigma$ and $\Sigma^\star$, contained in $G(2,V_4)$, with $G(2,V_4)\cap X=\Sigma\cup \Sigma^\star$, that move in  positive-dimensional families. Assume $G(2,V'_4)\cap X=\Sigma'\cup \Sigma^{\prime\star}$, with $V_4\ne V'_4$. Since $\Sigma$ is a linear section of $G(2,V_4)$, the intersection  $\Sigma\cap \Sigma'$ is a linear section of $G(2,V_4)\cap G(2,V'_4)=G(2,V_4\cap V'_4)$, a 2-plane. Since $\Sigma$ is irreducible, it is therefore either a line, a point, or the empty set. If $\Sigma'$ is a deformation of $\Sigma$ in $X$, we have  $\Sigma\cdot \Sigma'=(\Sigma^2)_X=3$, hence the intersection $\Sigma\cap \Sigma'$ must be a line $L\subset G(2,V_4\cap V'_4)$. Similarly, since $ \Sigma^\star\cdot \Sigma'= (\sigma_{1,1}\vert_X- \Sigma)\cdot \Sigma'=-2$, the intersection $\Sigma^\star\cap \Sigma'$ is a line $L^\star\subset G(2,V_4\cap V'_4)$. We also have
\begin{eqnarray}
G(2,V_4\cap V'_4)\cap X&=&(\Sigma\cup \Sigma^\star)\cap (\Sigma'\cup \Sigma^{\prime\star})\label{eea}\\
&=&L\cup L^\star\cup (\Sigma\cap \Sigma^{\prime\star})\cup (\Sigma^\star\cap \Sigma^{\prime\star}).\nonumber
\end{eqnarray}
  {\em If we assume
  $L\ne L^\star$,} since the set in (\ref{eea}) is the intersection of the 2-plane $G(2,V_4\cap V'_4)$ with a hyperplane and a quadric,  it must be $L\cup L^\star$. In that case, since $L^\star\not\subset \Sigma$, we have $\Sigma\cap \Sigma^{\prime\star}=L$ and similarly, 
  $\Sigma^\star\cap \Sigma^{\prime\star}=L^\star$. Then, $\Sigma'\cap\Sigma= \Sigma^{\prime\star} \cap\Sigma$ so, by switching the roles of $\Sigma$ and $\Sigma'$, we obtain a contradiction. Hence,
  $$L=L^\star\subset \Sigma\cap\Sigma^\star\cap G(2,V_4\cap V'_4).$$
  The projective lines in $L$ then sweep out the 2-plane $\P(V_4\cap V'_4)\subset \P(V_5)$, which is therefore independent of $V'_4$, hence of $\Sigma'$; call it $\P(V_3)$. If $X$ contains infinitely many quadric surfaces, they must occur as components of $G(2,V'_4)\cap X$ for all $V_3\subset V'_4\subset V_5$. 
  The union
  $$\bigcup_{V_3\subset V'_4\subset V_5}G(2,V'_4)\cap X  $$
  is then the intersection of $X$ with  the Schubert cycle of lines that meet the fixed 2-plane $\P(V_3)$. This is a hyperplane section of $X$, which is irreducible by Lefschetz theorem. This implies that $\Sigma$ and $\Sigma^\star$ are interchanged by monodromy; but this is impossible since $(\Sigma^2)_X=3$, whereas $(\Sigma\cdot \Sigma^\star)_X=-2$.
 \end{rema}

     \subsection{Fourfolds containing a cubic scroll (divisor $\cD_{12}$)}\label{cub}
We consider rational cubic scroll surfaces obtained as smooth hyperplane sections of the image of a morphism  $\P(V_2)\times\P(V_3)\to G(2,V_5)$, where $V_5=V_2\oplus V_3$; their class  in $G(2,V_5)$ is  $\sigma_1^2\cdot\sigma_2=2\sigma_{3,1}+\sigma_{2,2}$.

     \begin{prop}
      The closure $\overline \cX_{\hbox{\tiny\rm cubic scroll}}\subset \cX_{10}$
  of the
    family  of fourfolds containing    a cubic scroll surface  is the irreducible component of $\wp^{-1}( \cD_{12})$ that contains the family  $\cX_{\rho\hbox{\tiny\rm -plane}}$.
  \end{prop}

\begin{proof}
Let us count parameters. We have $6+6=12$ parameters for the choice of $V_2$ and $V_3$, hence {\em a priori} 12 parameters for  cubic scroll surfaces in the isotropic Grassmannian $G_\omega(2,V_5)$. However, one checks that there is a 1-dimensional family of $V_3$ which all give the same cubic scroll, so there are actually only 11 parameters. Then, for $X$ to contain a given cubic scroll $F$ represents $h^0(F,\cO_F(2,2))=12$ conditions. It follows that $\cX_{\hbox{\tiny\rm cubic scroll}}$ is irreducible of   codimension $12-11=1$ in $\cX_{10}$.

Let $F\subset X$ be a cubic scroll.
Since $K_F $ has type $(-1,-2)$,
 we obtain $(F)^2_X=4$ from (\ref{S2}). From (\ref{S3}), we obtain $d=12$  and  we are in $\cD_{12}$. The family $\cX_{\hbox{\tiny cubic scroll}}$   is therefore a component of the hypersurface $\wp^{-1}(\cD_{12})$.
 
 In the degenerate situation where $V_4=V_2+V_3$ is a hyperplane, the associated rational cubic scroll is contained in  $G(2,V_4)$ and is a cubic scroll surface as in the comment right before Remark  \ref{planf}. It follows that $\cX_{\rho\hbox{\tiny\rm -plane}}$ is contained in the closure of $\cX_{\hbox{\tiny\rm cubic scroll}}$.
 \end{proof}

     \subsection{Fourfolds containing a quintic del Pezzo  surface  (divisor $\cD''_{10}$)}\label{qdp}
We consider quintic del Pezzo surfaces  obtained as the intersection of $ G(2,V_5)$ with a   $\P^5$; their class is  $\sigma_1^4=3\sigma_{3,1}+2\sigma_{2,2}$ in $G(2,V_5)$. Fourfolds of type $\cX_{10}$ containing such a surface were  already  studied by  Roth (\cite{rot}, \S4).

  \begin{prop}
      The closure $\overline \cX_{\hbox{\tiny\rm quintic}}\subset \cX_{10}$
  of the
    family  of fourfolds containing     a quintic del Pezzo   surface is the irreducible component of $\wp^{-1}( \cD''_{10})$ that contains
  $\cX_{\sigma\hbox{\tiny\rm -plane}}$.
  
     A general member of $\cX_{\hbox{\tiny\rm quintic}}$ is rational.
  \end{prop}

\begin{proof}
Let us count parameters. We have $\dim G(5,\P^8)=18$ parameters for the choice of the $\P^5$ that defines a del Pezzo surface $T$. Then, for $X$ to contain a given quintic del Pezzo surface $T$  represents $h^0(\P^5,\cO(2))- h^0(\P^5,\cI_T(2))=21-5=16 $ conditions. 

Since $h^0(\P^8,\cI_X(2))=6=h^0(\P^5,\cI_T(2))+1$,   there exists then a unique (non-Pl\"ucker) quadric $Q\subset \P^8$ containing $X$ and $\P^5$. This quadric has     rank $\le 6$, hence it is a cone with vertex a 2-plane over a (in general) smooth quadric in $\P^5$. Such a quadric 
  contains two 3-dimensional families of 5-planes. The intersection of such a 5-plane with $X$ is, in general, a quintic del Pezzo surface, hence $X$ contains  (two) 3-dimensional families of quintic del Pezzo surfaces. It follows that  $\cX_{\hbox{\tiny\rm quintic}}$ has codimension $16-18+3=1$ in $\cX_{10}$.

Let $T\subset X$ be a quintic del Pezzo surface.
From (\ref{S2}), we obtain $(T)^2_X=5$, and from (\ref{S3}),  $d=10$. Since $\sigma_{1,1}\cdot T$ is odd, we are in $\cD''_{10}$.
  The family $\cX_{\hbox{\tiny\rm quintic}}$  is therefore a component of the divisor $\wp^{-1}(\cD''_{10})$.

The lattice spanned by $H^4(G(2,V_5),\Z)$ and $[T]$ in $H^4(X,\Z)$ is the same as for fourfolds containing a $\sigma$-plane $P$, and $[T]= \sigma_2\vert_X-[P]$. We will now explain this fact geometrically.

If $X$ contains a quintic del Pezzo   surface,  we saw that $X$ is contained in a   (non-Pl\"ucker)    quadric   $Q\subset \P^8$ of rank $\le 6$.
 Conversely, if $X$ is contained in such a quadric, this quadric   contains  5-planes and the intersection of such a 5-plane with $X$ is, in general, a quintic del Pezzo surface.

If follows that   $\cX_{\hbox{\tiny quintic}}$   has same closure in $\cX_{10}$ as the set of $X$ contained in a non-Pl\"ucker  rank-6 quadric $Q$. When the vertex of $Q$ is contained in $X$, it is a $\sigma$-plane, hence ${\overline\cX}\!_{\hbox{\tiny quintic}}$  contains $\cX_{\sigma\hbox{\tiny\rm -plane}}$.

Finally, note after \cite{rot}, \S5.(5), that   the general fibers of the projection $X\dra \P^2$ from $\langle T\rangle$ are again degree-5 del Pezzo surfaces (they are residual surfaces to $T$ in the intersection of $X$ with a 6-plane $\langle T,x\rangle$, and this intersection is contained in $\langle T,x\rangle\cap Q$, which is the union of two hyperplanes). It follows from a theorem of Enriques that $X$ is rational (\cite{enr}, \cite{shb}).\end{proof}

         \subsection{Nodal fourfolds (divisor $\cD_8$)} \label{sno}
Let $X$ be a general prime {\em nodal} Fano fourfold of index 2 and degree 10. As in the 3-dimensional case (\cite{dim2}, Lemma 4.1), $X$ is the intersection 
of a smooth ${G^\omega}:=G(2,V_5)\cap \P^8$ with a nodal quadric $Q$, singular at a general point $O\in 
{G^\omega}$. 

One checks that, as in the case of cubic fourfolds (see \cite{voi}, \S4; \cite{has}, Proposition 4.2.1), the limiting Hodge structure is pure, and the period map extends to the moduli stack ${\overline\cX}\!_{10}$ of our fourfolds with at most one node as
$$\overline\wp:{\overline\cX}\!_{10}\to \cD
.$$

  \begin{prop}
     The closure $\overline \cX_{\hbox{\tiny\rm nodal}}\subset {\overline\cX}\!_{10}$
  of the
    family  of nodal fourfolds    is an irreducible component  of $\overline\wp^{-1}( \cD_8)$.
   \end{prop}

\begin{proof}
If  $\widetilde X\to X$ is the blow-up of $O$, the (pure)  limiting Hodge structure is  the direct sum of $\langle\delta\rangle$, where $\delta$ is the
  vanishing cycle, with self-intersection 2, and   $H^4(\widetilde X,\Z)$.
   In the basis $(\sigma_{1,1}\vert_X, \sigma_2\vert_X-\sigma_{1,1}\vert_X, \delta)$, the corresponding lattice $K$ has intersection matrix 
   $\begin{pmatrix}2&0&0\\0&2&0\\0&0&2
   \end{pmatrix}$. The discriminant is $8$ and we are in $\cD_8$.

The point $O$ defines a pencil of Pl\"ucker quadrics, 
singular at $O$, and the image ${G^\omega}_O$ of ${G^\omega}$ by the projection $p_O:\P^8\dra \P^7_O$ is the base-locus of a pencil of rank-6 quadrics (see \cite{dim2}, \S3).  One checks that ${G^\omega}_O$ contains the 4-plane $\P^4_O:=p_O(\T_{{G^\omega},O})$ and that ${G^\omega}_O$  is singular along a cubic surface contained in $\P^4_O$. If $\widetilde \P^7_O\to \P^7_O$ is the blow-up of $\P^4_O$, the strict transform $\widetilde {G^\omega}_O\subset \widetilde \P^7_O\subset   \P^7_O\times \P^2$ of ${G^\omega}_O$ is smooth and   the projection $\widetilde {G^\omega}_O\to \P^2$ is a $\P^3$-bundle (this can be checked by explicit computations as in \cite{dim2}, \S9.2).

 The image $X_O:=p_O(X )$  
is thus the base locus in $\P^7_O$ of a net of quadrics $\P$, containing 
a special line of rank-6 Pl\"ucker quadrics. The strict transform  $\widetilde X_O\subset  \widetilde {G^\omega}_O$     of $X_O$ is smooth. The induced projection $\widetilde X_O\to \P^2$ is a quadric bundle, with discriminant   a smooth sextic curve $\Gamma_6^\star\subset \P^2$ (compare with \cite{dim2}, Proposition 4.2) and associated double cover $S\to \P^2$ ramified along $\Gamma_6^\star$. It follows that $S$ is a K3 surface with a degree-2 polarization. By \cite{las}, Theorem II.3.1, there is an exact sequence
$$0\lra H^4(\widetilde X_O,\Z)_0\stackrel{\Phi}{\lra}H^2(S,\Z)_0(-1)\lra \Z/2\Z\lra 0.
$$
Both desingularizations
   $\widetilde X\to X_O$ and $\widetilde X_O\to  X_O$ are  small and their fibers all have dimension $\le 1$; by \cite{fuwang}, Proposition 3.1, the graph of the rational map  $\widetilde X\dra \widetilde X_O$ induces an isomorphism $ H^4(\widetilde X_O,\Z)\isomto H^4(\widetilde X,\Z)$ of polarized Hodge structures. 
This is also the
   non-special cohomology  $K^\bot$, which therefore has index 2 in   $H^2(S,\Z)_0(-1)$.
   
   When $X$ is general, so is $S$ among degree-2 K3 surfaces, hence the image  
 $  \overline\wp(\cX_{\hbox{\tiny\rm nodal}})$ has dimension 19. It follows that 
   $\cX_{\hbox{\tiny\rm nodal}}$ is   an irreducible component  of $\overline\wp^{-1}( \cD_8)$. 
  \end{proof}

 \section{Construction of special   fourfolds}\label{cons}
 
 Again following Hassett (particularly \cite{has}, \S4.3), we construct special fourfolds with given discriminant. Hassett's idea was to construct, using the surjectivity of the (extended) period map for K3 surfaces, nodal cubic fourfolds whose  Picard group also contains a rank-2 lattice with discriminant $d$ and to smooth them using the fact that the period map remains a submersion on the nodal locus (\cite{voi}, p.\ 597). This method should work in our case, but would require first to make the construction of \S\ref{sno} of a nodal fourfold $X$ of type $\cX_{10}$ from a given degree-2 K3 surface more explicit, and second to prove that the extended period map remains submersive at any point of the nodal locus.
 
 We prefer here to use the simpler construction of \S\ref{splane} to prove the following.
 
 \begin{theo}\label{th81}
 The image of the period map $ \wp: \cX_{10}^0\to \cD$ meets all divisors $\cD_d$, for
 $d\equiv 0\pmod4$ and $d\ge 12$, and all divisors $\cD'_d$ and $\cD''_d$, for
 $d\equiv 2\pmod8$ and $d\ge 10$, except possibly $\cD''_{18}$.
  \end{theo}
  
Actually, the divisor $\cD''_{18}$   also   meets the  image of the period map: in a forthcoming article, we construct birational transformations  that take elements of $\wp^{-1}(\cD'_d)$ to elements of $\wp^{-1}(\cD''_d)$.
 
 \begin{proof}
Our starting point is   Lemma 4.3.3 of \cite{has}: let $\Gamma$ be a  rank-2 indefinite even lattice containing a primitive   element   $h $  with $h^2=10$, and assume there is no $c\in \Gamma $ with 
 \begin{itemize}
\item  either $c^2=-2$ and $c\cdot h=0$;
\item  or $c^2=0$ and $c\cdot h=1$; 
\item  or $c^2=0$ and $c\cdot h=2$.
 \end{itemize}
  Then there exists a K3 surface $S$ with $\Pic(S)= \Gamma $ and $h$ is   very ample
on $S$, hence embeds it in $\P^6$. Assuming moreover that $S$ is not trigonal, \eg, that there are no classes $c\in \Gamma $ with $c^2=0$ and $c\cdot h=3$, it has   Clifford index 2 and is therefore obtained as the intersection of a Fano threefold $Z:=G(2,V_5)\cap \P^6$ with a quadric 
 (\cite{muk4}, (3.9); \cite{jk}, Theorem 10.3 and Proposition 10.5).

In particular, $S$ is
 an intersection of quadrics, hence  its projection   from a general point $p$ of $S$ is a (degree-9) smooth surface   $\widetilde S_p\subset \P^5$.
 
 We want to show that $Z$ is smooth and that the projection  $\widetilde Z_p\subset \P^5$ of $Z$ from $p$   is contained in a smooth quadric. First of all,  if $\Pi\subset \P(\wedge^2V_5^\vee) $ is the 2-plane of hyperplanes that cut out $\P^6$ in $\P(\wedge^2V_5)$,  one has (\cite{pv}, Corollary 1.6) 
  $$\Sing (Z )= \Pi^\bot \cap \bigcup_{[\omega]\in\Pi } G(2,\Ker(\omega )).$$
Since $S$ is smooth, $\Sing(Z)$ is 0-dimensional. If $\Sing(Z)\ne\vide$, some $[\omega_0]\in \Pi$ must have rank 2. The (one-dimensional) kernels of $\omega\vert_{\Ker(\omega_0)}$, when $[\omega]$ describes $\Pi\moins\{[\omega_0]\}$, sweep out  a $V_2\subset \Ker(\omega_0)$, and one checks that $Z$ contains a family of lines through $[V_2]$, parametrized by a rational cubic curve. The intersection  of the cone swept out by these lines and the quadric that defines $S$ in $Z$ is a sextic curve of genus 2 in $S$. Its class $c'$ thus satisfies $c^{\prime 2}=2$ and $c'\cdot h=6$, hence $(h-c')^2=0$ and $(h-c')\cdot h=4$.

So if we assume finally that there are no classes $c\in \Gamma $ with $c^2=0$ and $c\cdot h=4$, the threefold $Z$  is  smooth. 
  In suitable coordinates $(x_{i})_{0\le i\le 4}$ on $V_5$, inducing coordinates $(x_{ij})_{0\le i<j\le 4}$ on $\wedge^2V_5$, it can be defined in $G(2,V_5)$ by the equations
$$x_{12}-x_{03}=x_{02  }-x_{14}=x_{01}-x_{24}+x_{23}=0 
 $$
  (\cite{ili}, Technical Lemma (2.5.1)). 
  
  It is then known (\cite{pv}, Theorem 7.5 and Proposition 7.6) that  
  $\Aut(Z)$ is isomorphic to  $\PGL(2,\C)$ and acts on $Z$ with 3 (rational) orbits of respective dimensions 3, 2, and 1.   Since $S$ is not rational,    it must meet the open orbit. Take $p\in S$ in that orbit. 
 As in \cite{ili}, \S3.1, we may assume $p=e_{34}$ and we take   $(x_{03},x_{13},x_{23},x_{04},x_{14},x_{24})$ as homogeneous coordinates on $\P^6$.
 
   It is   classical (\cite{pv}, \S7) that $Z$ contains 3 lines passing through $p$, so that the   image  $\widetilde Z_p\subset \P^5$ of the projection of $Z$  from $p$ has exactly 3 singular points, which are also on the smooth surface $\widetilde S_p$.  The    degree-4 fourfold $\widetilde Z_p $ is the base-locus of the pencil generated by the rank-5 (restrictions of the) Pl\"ucker quadrics
 $$\Omega_{e_3}: \bigl((x_{24}-x_{23})x_{24}-x_{14}^2+x_{03}x_{04}=0 \bigr)\quad{\rm and}\quad
 \Omega_{e_4}: \bigl((x_{24}-x_{23})x_{23}-x_{13}x_{14}+x_{03}^2=0 \bigr)
 $$
  (\cite{ili}, \S3.1). One checks that the quadric $Y\subset \P^5$ defined by the sum of these two equations is smooth.
Consider the blow-up $\widetilde Y\to Y$ of 
 $\widetilde S_p$. The inverse image $E\subset \widetilde Y$ of $\widetilde Z_p$ is then a small resolution, which is isomorphic to the blow-up of $p$ in $Z$.
 
 The morphism $E\to \widetilde Z_p$  is in particular independent of the choice of $S$, $p$, and $Y$ and it follows from the description of the general case in the proof of Proposition \ref{prop71} that $E$ is a $\P^1$-bundle over $\P^2$. More precisely, the linear system $|H|$ on $\widetilde Y$ given by cubics containing $\widetilde S_p$  induces on $E$ a morphism $E\to \P^2$ with $\P^1$-fibers (in the notation of that proof, $E$ is the exceptional divisor of the blow-up $\widetilde X\to X$ of the plane $P$).
 
The linear system $|H|$  is base-point-free and injective outside of $E$ (because $\widetilde Z_p$ is a quadratic section of $Y$ which contains $\widetilde S_p$)  and base-point-free on $E$ as we just saw. Since $H^4=10$ and $h^0( \widetilde Y, H)=9$, it defines a birational morphism $\widetilde Y\thra X\subset  \P^8 $ which maps $E$ onto a 2-plane $P\subset X$.
 This morphism is one of the two $K_{\widetilde Y}$-negative extremal contractions of $\widetilde Y$ (the other one being the blow-up $\widetilde Y\to Y$); its fibers all have dimension $\le 1$, hence  $X$ is {\em smooth} and the contraction is the blow-up of $P$ (\cite{am}, Theorem 4.1.3).
 
 It is then easy to check that $X$ is a (special) Fano fourfold  of type $\cX_{10}^0$ containing $P$ as a $\sigma$-plane. As explained in  the proof of Proposition \ref{prop71}, its non-special cohomology is isomorphic to the 
primitive cohomology of $S$.
 
 \smallskip
 
 We will now apply this construction with various lattices $\Gamma$ to produce examples of smooth fourfolds $X$ which will all be in $\cX_{\sigma\hbox{\tiny\rm -plane}}$, hence with period point in $\cD''_{10}$, but whose lattice $H^{2,2}(X)\cap H^4(X,\Z)$ will contain other sub-lattices of rank 3 with various discriminants.

Apply first Hassett's lemma with the rank-2 lattice $\Gamma $ with matrix $\begin{pmatrix}10&0\\0&-2e
   \end{pmatrix}$ in a basis $(h,w)$. When $e>1$, the conditions we need on $\Gamma $ are satisfied and we obtain a K3 surface $S$ and  a smooth fourfold $X$ such that $H^4(X,\Z)\cap H^{2,2}(X)$ contains a lattice $K_{10}=\langle u,v,w''_{10}\rangle$,
   with matrix 
   $\begin{pmatrix}2&0&0\\0&2&1\\0&1&3
   \end{pmatrix}$ and discriminant 10. Moreover, $H^2(S,\Z)_0(-1)\isom K_{10}^\bot$ as polarized integral Hodge structures.  The  element $w\in \Gamma\cap H^2(S,\Z)_0$ corresponds to $w_X\in K_{10}^\bot\cap H^{2,2}(X)$, and $w_X^2=-w^2=2e$. 
   Therefore, $H^4(X,\Z)\cap H^{2,2}(X)$ is the lattice $\langle u,v,w''_{10},w_X\rangle$, with matrix
   $$\begin{pmatrix}2&0&0&0\\0&2&1&0\\0&1&3&0\\0&0&0&2e
   \end{pmatrix}.
   $$
It contains the lattice $\langle u,v,w_X\rangle$, with matrix 
   $\begin{pmatrix}2&0&0\\0&2&0\\0&0&2e
   \end{pmatrix}$. Therefore, the period of $X$ belongs to $\cD_{8e}$, and this  proves the theorem when $d\equiv 0\pmod{8}$.
   
 It also  contains the lattice $\langle u,v,w''_{10}+w_X\rangle$, with matrix
   $\begin{pmatrix}2&0&0\\0&2&1\\0&1&2e+3
   \end{pmatrix}
   $
   and discriminant $8e+10$, hence we are also in $\cD''_{8e+10}$.

   Now let $e\ge 0$ and apply Hassett's lemma with the  lattice $\Gamma $ with matrix $\begin{pmatrix}10&5\\5&-2e
   \end{pmatrix}$ in a basis $(h,g)$. The orthogonal of $h$ is spanned by $w:=h-2g$. One checks that primitive classes $c\in \Gamma $ such that $c^2=0$ satisfy $c\cdot h\equiv 0 \pmod 5$. All the conditions we need are thus satisfied and we obtain a K3 surface $S$ and  a smooth fourfold $X$ such that $H^4(X,\Z)\cap H^{2,2}(X)$ contains a lattice $K_{10}$ of discriminant 10 and $H^2(S,\Z)_0(-1)\isom K_{10}^\bot$ as polarized Hodge structures. Again, $w$ corresponds to $w_X\in K_{10}^\bot\cap H^{2,2}(X)$ with $w_X^2=-w^2=8e+10$. Set
   $$K:=(\Lambda_2\oplus \Z w_X)^{\rm sat}.$$
   To compute the discriminant of $K$, we need to know the ideal $w_X\cdot \Lambda$. As in the proof of Proposition \ref{p53}, let $w_{10}$ be a generator of $K_{10}\cap \Lambda$; it satisfies $w_{10}^2=10$. Then $K_{10}^\bot\oplus \Z w_{10}$ is a sublattice of  $\Lambda$ and, taking discriminants, we find that the index is 5. Let $u$ be an element of $\Lambda$ whose class generates the quotient. We have
   $$w_X\cdot \Lambda=\Z w_X\cdot u+w_X\cdot (K_{10}^\bot\oplus \Z w_{10})
   =\Z w_X\cdot u+w_X\cdot K_{10}^\bot =\Z w_X\cdot u+w\cdot H^2(S,\Z)_0.
   $$
  One checks directly on the K3 lattice that  $w\cdot H^2(S,\Z)_0=2\Z$. Since $5u\in K_{10}^\bot\oplus \Z w_{10}$, we have $5 w_X\cdot u\in 2\Z$, hence $w_X\cdot u\in 2\Z$. All in all, we have proved $w_X\cdot \Lambda = {2\Z}$ hence the proof of Proposition \ref{p53} implies that the discriminant of $K$ is $w_X^2=8e+10$. Therefore, the period of $X$ belongs to $\cD_{8e+10}$. 
  
  Since the period of $X$  is in $\cD''_{10}$,  we saw in the proof of Proposition \ref{p53} that $w_{10}'':=\frac12(v+w_{10})$ is in $H^4(X,\Z)$. Similarly,
   either $w'_X:=\frac12(u+w_X)$ or $w''_X:=\frac12(v+w_X)$ is in $K$. Taking intersections with $w''_{10}$ (and recalling $w_X\cdot w_{10}=0$ and $v\cdot w_{10}=1$), we see that we are in the first case, hence the  period of $X$ is actually in  $\cD'_{8e+10}$. 
   
   More precisely, $H^4(X,\Z)\cap H^{2,2}(X)$ is the lattice $\langle u,v,w''_{10},w'_X\rangle$, with matrix
   $$\begin{pmatrix}2&0&0&1\\0&2&1&0\\0&1&3&0\\1&0&0&2e+3
   \end{pmatrix}.
   $$
This lattice also contains the lattice $\langle u,v,w''_{10}+w'_X\rangle$, with matrix
   $\begin{pmatrix}2&0&1\\0&2&1\\1&1&2e+6
   \end{pmatrix}
   $
   and discriminant $8e+20$, hence we are also in $\cD_{8e+20}$.

Since we know from \S\ref{rplane}  that the periods of some smooth fourfolds $X$ 
 of type $\cX^0_{10}$ lie in $\cD_{12}$, this proves the theorem when $d\equiv 4\pmod{8}$.
  \end{proof}

        \section{Summary and open questions}\label{last}
        
        We summarize the results of \S\ref{exam} in the following diagrams (where the tags next to the arrows describe the general fibers):
        $$
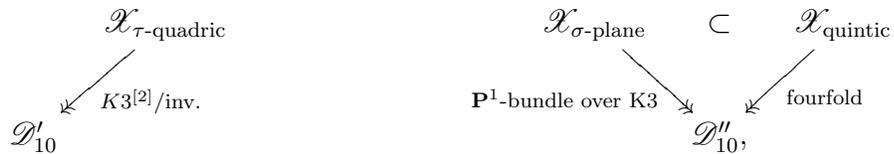
\xymatrix @C=10pt
{
&&\cX_{\tau\hbox{\tiny\rm -quadric}}\ar@{->>}[dl]^{\hbox{\tiny $K3^{[2]}/\text{inv.}$}}&&&&&&&
\cX_{\sigma\hbox{\tiny\rm -plane}}\ar@{->>}[dr]_{\hbox{\tiny $\P^1$-bundle  over K3}}&\subset& \cX_{\hbox{\tiny\rm quintic}}\ar@{->>}[dl]^{\hbox{\tiny fourfold}}\\
&\cD'_{10}&&&&&&&&&\cD''_{10},
}$$
where the general members of each of  these families are all rational, and
        $$
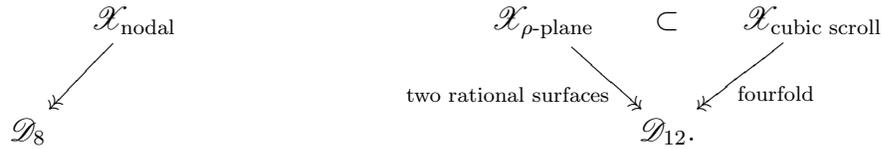
\xymatrix @C=10pt
{
&&\cX_{\hbox{\tiny\rm nodal}}\ar@{->>}[dl]&&&&&&&\cX_{\rho\hbox{\tiny\rm -plane}}\ar@{->>}[dr]_-{\hbox{\tiny two rational surfaces}}&\subset&\cX_{\hbox{\tiny\rm cubic scroll}}\ar@{->>}[dl]^{\hbox{\tiny fourfold}}
\\
&\cD_8&&&&&&&&&\cD_{12}.
}$$

In the first two of these diagrams, all fourfolds in a given (general) fiber of the period map are of course birationally isomorphic. We think this should be a general fact. More precisely,  all fourfolds with the same period should be related by an analog of the conic transformations discussed in \cite{dim} in the case of threefolds.
        
  It would   be very interesting, as  Laza did for cubic fourfolds (\cite{laza}, Theorem 1.1),
 to determine the exact image   in the period domain $\cD$ of the period map for our fourfolds.

  \begin{ques}
Is the image   of the   period map  equal to  $\cD\moins \cD_2\moins \cD_4\moins \cD_8$? \end{ques}
 
Answering this question seems   far from our present possibilities; to start with, inspired by the results of \cite{has}, one could ask  (see Theorem \ref{th81}) whether the image   of the period map is disjoint from the hypersurfaces $\cD_2$, $\cD_4$, and $\cD_8$.

\end{document}